\documentclass[11pt,a4paper]{amsart}

\usepackage{a4wide}

\usepackage[english]{babel}

\usepackage[pdftex]{graphicx}

\usepackage[all]{xy}
\usepackage{amsmath,amsthm,amssymb,enumerate}
\usepackage{latexsym}
\usepackage{amscd}
\usepackage{tikz-cd}

\usepackage[colorlinks=true]{hyperref}

\newtheorem{theorem}{Theorem}[section]
\newtheorem{thm}[theorem]{Theorem}
\newtheorem{definition}[theorem]{Definition}
\newtheorem{defn}[theorem]{Definition}
\newtheorem{prop}[theorem]{Proposition}
\newtheorem{proposition}[theorem]{Proposition}

\newtheorem{cor}[theorem]{Corollary}
\newtheorem{lemma}[theorem]{Lemma}

\newtheorem{remark}[theorem]{Remark}
\newtheorem{rem}[theorem]{Remark}
\newtheorem{example}[theorem]{Example}

\newcommand{\cat}{\mathcal}

\newcommand{\R}{\mathbb R}
\newcommand{\Q}{\mathbb Q}
\newcommand{\Z}{\mathbb Z}
\newcommand{\N}{\mathbb N}
\newcommand{\C}{\mathbb C}
\newcommand{\T}{\mathbb T}

\newcommand{\CP}{\mathbb P}

\newcommand{\HH}{\mathbb H}
\newcommand{\ZZ}{\mathbb Z}
\newcommand{\QQ}{\mathbb Q}

\newcommand{\fk}{{\mathfrak{k}}}

\newcommand{\Xx}{{\cat X}}

\newcommand{\om}{{\omega}}
\newcommand{\vr}{{\varphi}}

\newcommand{\ga}{\gamma}
\def\hga{{\hat\gamma}}
\def\bga{{\bar\gamma}}
\def\D{{\mathfrak{D}}}
\def\maslov{{\mu_{\text{Maslov}}}}

\newcommand{\cp}{{{\C P}\,\!}}

\newcommand{\rp}{{{\R\CP}\,\!}}

\newcommand{\balpha}{{\bar{\alpha}}}

\newcommand{\teta}{{\tilde{\eta}}}

\newcommand{\txi}{{\tilde{\xi}}}
\newcommand{\tnu}{{\tilde{\nu}}}

\newcommand{\id}{\mathit{id}}

\DeclareMathOperator{\Sp}{Sp}

\DeclareMathOperator{\conv}{conv}

\DeclareMathOperator{\Aff}{Aff}

\def\om{{\omega}}

\newcommand{\SU}{\text{SU}}

\def\rs{{\mu_\text{RS}}}

\def\vr{{\varphi}}

\begin{document}
\title[invariants of non-simply connected toric contact manifolds]
{On contact invariants of non-simply connected Gorenstein toric contact manifolds}

\author[M.~Abreu]{Miguel Abreu}
\address{Center for Mathematical Analysis, Geometry and Dynamical Systems,
Instituto Superior T\'ecnico, Universidade de Lisboa, 
Av. Rovisco Pais, 1049-001 Lisboa, Portugal}
\email{mabreu@math.tecnico.ulisboa.pt, macarini@math.tecnico.ulisboa.pt}

\author[L.~Macarini]{Leonardo Macarini}
 
\author[M.~Moreira]{Miguel Moreira}
\address {ETH Z\"urich, Department of Mathematics. R\"amistrasse 101, 8092 Z\"urich, Switzerland.}
\email{miguel.moreira@math.ethz.ch}

\thanks{Partially funded by FCT/Portugal through UID/MAT/04459/2020. MA and MM were  also
funded through project PTDC/MAT-GEO/1608/2014. MA and LM were also funded by CNPq/Brazil.  
The present work started as part of MA and LM activities within BREUDS, a research partnership 
between European and Brazilian research groups in dynamical systems, supported by an FP7 
International Research Staff Exchange Scheme (IRSES) grant of the European Union.
MM received funding from the European Research Council (ERC) under the European Unions Horizon 
2020 research and innovation programme (ERC-2017-AdG-786580-MACI).}

\begin{abstract}
The first two authors showed in~\cite{AM1} how the Conley-Zehnder index of any contractible periodic Reeb
orbit of a non-degenerate toric contact form on a good toric contact manifold with zero first Chern class, i.e.
a Gorenstein toric contact manifold, can be explicitly computed using moment map data. In this paper we show
that the same explicit method can be used to compute Conley-Zehnder indices of non-contractible periodic
Reeb orbits. Under appropriate conditions, the (finite) number of such orbits in a given free homotopy class and
with a given index is a contact invariant of the underlying contact manifold. We compute these invariants for two 
sets of examples that satisfy these conditions: $5$-dimensional contact manifolds that arise as unit cosphere bundles 
of $3$-dimensional lens spaces, and $2n+1$-dimensional Gorenstein contact lens spaces. As applications, we 
will see that these invariants can be used to show that diffeomorphic lens spaces might not be contactomorphic 
and that there are homotopy classes of diffeomorphisms of some lens spaces that do not contain any contactomorphism. 
Following a suggestion by one referee, we will also see that this type of applications can be proved alternatively by 
looking at the total Chern class of these canonical contact structures on lens spaces.
\end{abstract}



\maketitle

\section{Introduction}
\label{s:intro}

The first two authors showed in~\cite{AM1} how the Conley-Zehnder index of any contractible periodic Reeb
orbit of a non-degenerate toric contact form on a good toric contact manifold with zero first Chern class, i.e.
a Gorenstein toric contact manifold, can be explicitly computed using moment map data. In this paper we show
that the same explicit method can be used to compute Conley-Zehnder indices of non-contractible periodic
Reeb orbits, discuss several examples and give applications.

Recall that the contact homology degree, or Symplectic Field Theory (SFT) degree, of a non-degenerate
closed Reeb orbit is equal to the Conley-Zehnder index plus $n-2$, where throughout this paper contact 
manifolds have dimension $2n+1$. As shown in~\cite{AM1,AM2}, this degree is always even for the
closed Reeb orbits of any non-degenerate toric contact form. Hence, given a non-degenerate toric contact form 
$\alpha$ on a Gorenstein toric contact manifold $(M, \xi)$, with corresponding Reeb vector field $R_\alpha$,
each \emph{contact Betti number} $cb_j (M,\alpha)$, $j\in\Z$, defined by
\[
cb_j (M, \alpha) = \text{number of closed $R_\alpha$-orbits with contact homology degree $j$,}
\]
should be a contact invariant of $(M,\xi)$, the rank of its degree $j$ cylindrical contact homology,
Unfortunately, and despite recent foundational developments (e.g.~\cite{P1,P2}), cylindrical contact 
homology has not been proved to be a well defined invariant in the presence of contractible closed
Reeb orbits, even in this restricted context of Gorenstein toric contact manifolds. 

There are however at least two particular contexts that allow us to conclude that the contact Betti 
numbers are indeed contact invariants:
\begin{itemize}
\item[(i)] Gorenstein toric contact manifolds that have crepant (i.e. with zero first Chern class) toric 
symplectic fillings.
\item[(ii)] Gorenstein toric contact manifolds that have a non-degenerate toric contact form with all 
of its closed contractible Reeb orbits having Conley-Zehnder index strictly greater than $3-n$, i.e. 
SFT degree strictly greater than $1$.
\end{itemize}
Indeed, in both of these contexts we can use positive equivariant symplectic homology to conclude 
that the contact Betti numbers are contact invariants, which will be denoted by $cb_j (M, \xi)$ or just
$cb_j (M)$ when the contact structure is clear from the context. 
For (i), one considers the positive equivariant symplectic homology of the filling and recent work of 
McLean and Ritter~\cite{MR}, which uses previous work of Kwon and van Koert~\cite{KvK} 
(see~\cite[Remark 1.4]{AM2} and~\cite[section 2]{AMGK}). For (ii), one considers the
positive equivariant symplectic homology of the symplectization and the work of Bourgeois and 
Oancea~\cite[section 4.1.2]{BO17}. In this case, we can also consider the decomposition
of positive equivariant symplectic homology induced by homotopy classes $[\gamma]$ in the (cyclic) fundamental 
group of these Gorenstein toric manifolds to get possibly finer invariants of the their (co-oriented) contact
structures:
\[
cb_j ^{[\gamma]} (M, \xi) = \text{number of closed $R_\alpha$-orbits in $[\gamma]$
with contact homology degree $j$.}
\]

The examples discussed in this paper are all within context (ii) above, with some being also within
context (i).

The first set of examples consists of the $5$-dimensional contact manifolds that arise as unit cosphere
bundles of $3$-dimensional lens spaces. We will show that they are Gorenstein toric contact manifolds,
fit within context (ii) and compute their contact Betti numbers, including the decomposition induced by
homotopy classes in the fundamental group. Note that since any $5$-dimensional Gorenstein
toric contact manifold has a crepant toric symplectic filling, these examples also fit within context (i).

The second set of examples consists of $2n+1$-dimensional lens spaces 
$$L^{2n+1}_p (\ell_0, \ell _1, \ldots, \ell_n)$$
with canonical contact structure, arising as the quotient of the standard contact sphere $S^{2n+1}\subset\C^{n+1}$ 
by the $\Z_p$-action with weights $\ell_0, \ell _1, \ldots, \ell_n\in\Z$. Such an action is free when the weights are 
coprime with $p$ and the resulting smooth contact lens space has zero first Chern class if and only if the sum of 
the weights is zero mod $p$. Being finite quotients of the standard contact sphere, all these Gorenstein toric contact 
lens spaces can easily be seen to fit context (ii), although most of them do not fit context (i). We will give an explicit 
method to compute their contact Betti numbers and their decomposition by homotopy classes of the fundamental group 
(see Proposition~\ref{prop:lensBetti}), and illustrate how these invariants can be used to prove the following type of 
results.

\begin{prop} \label{prop:ineq}
The diffeomorphic lens spaces 
\[
L^{13}_5 (1, 1,1,1,2,-2,1)  \quad\text{and}\quad L^{13}_5 (1, -1,-1,-1,-2,-2,1)
\]
have inequivalent canonical contact structures, both with zero first Chern class.
\end{prop}

\begin{prop} \label{prop:auto}
The lens space $L^{15}_5 (1,1,1,2,-2,-2,-2,1)$ has an homotopy class of orientation preserving
diffeomorphisms that does not contain any contactomorphism of its canonical contact structure. 
In fact, it has an orientation preserving diffeomorphism that acts by multiplication by $2\in (\Z_5)^\ast$ 
on its fundamental group $\Z_5$, while any contactomorphism of its canonical contact structure
acts by multiplication by $\pm 1\in (\Z_5)^\ast$.
\end{prop} 

\begin{remark} \label{rmk:referee}
Following a very relevant suggestion by one referee, it turns out that this type of results can also be proved
by looking at the total Chern class of these canonical contact structures on lens spaces. In fact, we have
that  
$$H^{2j} (L^{2n+1}_p (\ell_0, \ell _1, \ldots, \ell_n); \Z) \cong \Z_p\,, \ \text{for} \ j=1,\ldots,n\,,$$ 
and the Chern classes of the canonical contact structure are given by (cf. Proposition~\ref{prop:totalChern}
in appendix~\ref{appendix})
\[
c_j = \sigma_j \!\!\!\!  \mod p\,, j=1,\ldots n\,,
\]
where $\sigma_j =$ $j$-th elementary symmetric polynomial of $\ell_0, \ell _1, \ldots, \ell_n$.
For example,
\[
c_1 = \sum_{k=0}^{n} \ell_k \mod p \quad\text{and}\quad c_n = \sum_{k=0}^{n+1} \frac{\ell_0 \cdot \cdots \cdot \ell_n}{\ell_k} \mod p\,.
\]
In particular, this gives
\[
c_6 (L^{13}_5 (1, 1,1,1,2,-2,1))  = 1 + 1 + 1 + 1 -2 + 2 + 1 = 5 = 0 \mod 5
\]
and
\[
c_6 (L^{13}_5 (1, -1,-1,-1,-2,-2,1)) = 1 -1 -1 -1 + 2 + 2 + 1 = 3 \neq 0 \mod 5\,,
\]
which implies Proposition~\ref{prop:ineq}. Regarding Proposition~\ref{prop:auto}, we have that
\[
c_7 (L^{15}_5 (1,1,1,2,-2,-2,-2,1)) = -1 -1 -1 + 2 -2 -2 -2 -1 = 2 \neq 0 \mod 5\,.
\]
Moreover, if a diffeomorphism of a lens space acts by multiplication by $k \in \Z_k^\ast$ on its fundamental group, 
then it acts by multiplication by $k^j \in \Z_k^\ast$ on its degree $2j$ cohomology. Hence, the orientation 
preserving diffeomorphism considered in Proposition~\ref{prop:auto} does not fix $c_7 (L^{15}_5 (1,1,1,2,-2,-2,-2,1))$,
since it acts on it by multiplication by $2^7 = 128 \equiv 3 \mod 5$, which immediately implies that there is no 
contactomorphism in its homotopy class.

Although there are lens spaces with canonical contact structure having total Chern class equal to zero, 
they do not provide relevant examples for the type of results considered in Propositions~\ref{prop:ineq} 
and~\ref{prop:auto}. In fact, as we will see in appendix~\ref{appendix}, at least for prime $p$ there are no examples 
of lens spaces where the  type of results considered in Propositions~\ref{prop:ineq} and~\ref{prop:auto} do not follow 
from total Chern class considerations, the reason being that the canonical contact structure of a lens space 
$L^{2n+1}_p (\ell_0, \ell _1, \ldots, \ell_n)$  is completely determined by the diffeomorphism type of the lens space 
and the total Chern class of the contact structure.

\end{remark} 

The paper is organized as follows.
\begin{itemize}
\item In section~\ref{s:Gorenstein} we recall fundamental facts about Gorenstein toric contact manifolds, including 
some (perhaps new) observations relating their fundamental groups with combinatorial properties of their
toric diagrams. We will also present some examples, describing in detail the ones mentioned above.
\item Section~\ref{s:trivialization} is the main technical section of the paper, where we show that the explicit method 
developed in~\cite{AM1}, to compute the Conley-Zehnder index of any contractible periodic Reeb orbit of a non-degenerate 
toric contact form on a Gorenstein toric contact manifold, can also be used for non-contractible periodic Reeb orbits.
\item In section~\ref{s:examples1} we compute the contact Betti numbers of unit cosphere bundles of $3$-dimensional 
lens spaces, including their decomposition induced by homotopy classes in the fundamental group.
\item In section~\ref{s:examples2}, we give an explicit method to compute the contact Betti numbers of
Gorenstein contact lens spaces and their decomposition by homotopy class of the fundamental group 
(Proposition~\ref{prop:lensBetti}), illustrate how it can be explicitly used in some examples, including Gorenstein 
prequantizations of $\cp^n$ (Example~\ref{ex:preq}), and prove Propositions~\ref{prop:ineq} and~\ref{prop:auto}.
\item Finally, appendix~\ref{appendix} is devoted to considerations regarding the total Chern class of the canonical contact
structure of a lens space and results of the type considered in Propositions~\ref{prop:ineq} and~\ref{prop:auto}.
\end{itemize}

\section{Gorenstein toric contact manifolds}
\label{s:Gorenstein}

In this section we recall fundamental facts about Gorenstein toric contact manifolds that will be needed in the paper,
including some (perhaps new) observations relating their fundamental groups with combinatorial properties of their
toric diagrams. We will also present some examples, describing in detail the ones that were already mentioned in 
the Introduction, i.e. $5$-dimensional contact manifolds that arise as unit cosphere bundles of $3$-dimensional lens 
spaces and $2n+1$-dimensional lens spaces.

\subsection{Toric diagrams}
\label{ss:diagrams}

In this subsection we closely follow the presentation in~\cite{AM2} to recall the $1$-$1$ correspondence between 
Gorenstein toric contact manifolds, i.e. good toric contact manifolds (in the sense of Lerman~\cite{Le1}) with 
zero first Chern class, and toric diagrams (defined below).

Via symplectization, there is a $1$-$1$ correspondence between co-oriented contact manifolds
and symplectic cones, i.e. triples $(W,\om,X)$ where $(W,\om)$ is a connected symplectic manifold
and $X$ is a vector field, the Liouville vector field, generating a proper $\R$-action
$\rho_t:W\to W$, $t\in\R$, such that $\rho_t^\ast (\om) = e^{t} \om$. A closed symplectic cone is a 
symplectic cone $(W,\om,X)$ for which the corresponding contact manifold $M = W/\R$ is closed.

A toric contact manifold is a contact manifold of dimension $2n+1$ equipped with an effective Hamiltonian
action of the standard torus of dimension $n+1$: $\T^{n+1} = \R^{n+1} / 2\pi\Z^{n+1}$. Also via symplectization,
toric contact manifolds are in $1$-$1$ correspondence with toric symplectic cones, i.e. symplectic cones
$(W,\om,X)$ of dimension $2(n+1)$ equipped with an effective $X$-preserving Hamiltonian $\T^{n+1}$-action,
with moment map $\mu : W \to \R^{n+1}$ such that $\mu (\rho_t (w)) = e^{t} \mu (w)$, for all $w\in W$ and $t\in\R$.
Its moment cone is defined to be $C:= \mu(W) \cup \{ 0\} \subset \R^{n+1}$.

A toric contact manifold is {\it good} if its toric symplectic cone has a moment cone with the following properties.
\begin{definition} \label{def:good}
A cone $C\subset\R^{n+1}$ is \emph{good} if it is strictly convex and there exists a minimal set 
of primitive vectors $\nu_1, \ldots, \nu_d \in \Z^{n+1}$, with 
$d\geq n+1$, such that
\begin{itemize}
\item[(i)] $C = \bigcap_{j=1}^d \{x\in\R^{n+1}\mid 
\ell_j (x) := \langle x, \nu_j \rangle \geq 0\}$.
\item[(ii)] Any codimension-$k$ face of $C$, $1\leq k\leq n$, 
is the intersection of exactly $k$ facets whose set of normals can be 
completed to an integral basis of $\Z^{n+1}$.
\end{itemize}
The primitive vectors $\nu_1, \ldots, \nu_d \in \Z^{n+1}$ are called the defining normals of the good cone $C\subset\R^{n+1}$.
\end{definition}
The analogue for good toric contact manifolds of Delzant's classification theorem for closed toric
symplectic manifolds is the following result (see~\cite{Le1}).
\begin{theorem} \label{thm:good}
For each good cone $C\subset\R^{n+1}$ there exists a unique closed toric symplectic cone
$(W_C, \om_C, X_C, \mu_C)$ with moment cone $C$.
\end{theorem}
The existence part of this theorem follows from an explicit symplectic reduction of the standard euclidean
symplectic cone $(\R^{2d}\setminus\{0\}, \omega_{\rm st}, X_{\rm st})$, where $d$ is the number of defining normals of the
good cone $C\subset\R^{n+1}$, with respect to the action of a subgroup $K\subset\T^d$ induced by the standard
action of $\T^d$ on $\R^{2d}\setminus\{0\} \cong \C^d \setminus\{0\}$. More precisely,
\begin{equation} \label{eq:defK}
K := \left\{[y]\in\T^d \mid \sum_{j=1}^d y_j \nu_j \in 2\pi\Z^{n+1}   \right\}\,,
\end{equation}
where $\nu_1, \ldots, \nu_d \in \Z^{n+1}$ are the defining normals of $C$, i.e. $K:= \ker (\beta)$ where $\beta :\T^d \to \T^{n+1}$
is represented by the matrix
\begin{equation} \label{eq:defBeta}
\left[\ \nu_1 \ | \ \cdots \ | \ \nu_d \  \right]\,.
\end{equation}
Depending on the context, which will be clear in each case, we will also denote by $\beta$ the map from $\Z^d$ to $\Z^{n+1}$  
represented by this matrix.

The Chern classes of a co-oriented contact manifold can be canonically identified with 
the Chern classes of the tangent bundle of the associated symplectic cone.
The following proposition gives a moment cone characterization of the vanishing of the first Chern class 
that is commonly used in toric Algebraic Geometry (see, e.g., \S $4$ of~\cite{BD}).
\begin{prop} \label{prop:c_1}
Let $(W_C, \omega_C,X_C)$ be a good toric symplectic cone.
Let $\nu_1,\ldots,\nu_d \in \Z^{n+1}$ be the defining normals of the corresponding moment cone 
$C\subset\R^{n+1}$. Then $c_1 (TW_C) = 0$ if and only if there exists $\nu^\ast \in (\Z^{n+1})^\ast$ 
such that
\[
\nu^\ast (\nu_j) = 1\,,\ \forall\ j=1,\ldots,d\,.
\]
\end{prop}
By an appropriate change of basis of the torus $\T^{n+1}$, i.e. an appropriate $SL(n+1,\Z)$
transformation of $\R^{n+1}$, this implies the following.
\begin{cor} \label{cor:c_1}
Let $(W_C, \omega_C,X_C)$ be a good toric symplectic cone with $c_1 (TW_C) = 0$. Then there 
exists an integral basis of $\T^{n+1}$ for which the defining normals $\nu_1,\ldots,\nu_d \in \Z^{n+1}$ 
of the corresponding moment cone $C\subset\R^{n+1}$ are of the form
\[
\nu_j = (v_j, 1)\,,\ v_j \in \Z^n\,,\ j=1,\ldots,d\,.
\]
\end{cor}

The next definition and theorem are then the natural analogues for Gorenstein toric contact
manifolds of Definition~\ref{def:good} and Theorem~\ref{thm:good}.
\begin{defn} \label{def:diagram}
A \emph{toric diagram} $D\subset\R^n$ is an integral simplicial polytope with
all of its facets $\Aff(n,\Z)$-equivalent to $\conv(e_1, \ldots, e_n)$, where
$\{e_1,\ldots,e_n\}$ is the canonical basis of $\R^n$.
\end{defn}
\begin{rem} \label{rem:diagram}
The group $\Aff(n,\Z)$ of integral affine transformations of $\R^n$ can be naturally identified
with the elements of $SL(n+1,\Z)$ that preserve the hyperplane $\left\{ (v,1)\mid v\in\R^n \right\} \subset\R^{n+1}$.
\end{rem}
\begin{thm} \label{thm:diagram}
For each toric diagram $D\subset\R^n$ there exists a unique Gorenstein toric contact
manifold $(M_D, \xi_D)$ of dimension $2n+1$.
\end{thm}

The $\T^{n+1}$-action associates to every vector $\nu \in \R^{n+1}$ a contact vector field $R_\nu \in \Xx (M_D, \xi_D)$. 
We will say that a contact form $\alpha_\nu \in \Omega^1 (M_D,\xi_D)$ is \emph{toric} if its Reeb vector field $R_{\alpha_\nu}$ 
satisfies
\[
R_{\alpha_\nu} = R_\nu \quad\text{for some $\nu\in\R^{n+1}$.}
\]
In this case we will say that $\nu\in\R^{n+1}$ is a \emph{toric Reeb vector} and that $R_\nu$ is a \emph{toric Reeb vector field}.

\begin{defn} \label{def:reeb}
A \emph{normalized} toric Reeb vector is a toric Reeb vector $\nu\in\R^{n+1}$ of the form
\[
\nu = (v,1) \quad\text{with $v\in\R^n$.}
\]
\end{defn}
\begin{prop}[{\cite{MSY} or \cite[Corollary 2.15]{AM2}}] \label{prop:reeb}
The interior of a toric diagram $D \subset \R^n$ parametrizes the set of normalized toric 
Reeb vectors on the Gorenstein toric contact manifold $(M_D, \xi_D)$, i.e. $\nu = (v,1)$ is
a normalized toric Reeb vector iff $v\in\textup{int} (D)$.
\end{prop}

\subsection{Fundamental group}
\label{ss:pi1}

Let $W_C$ be a good toric symplectic cone and $\nu_1,\ldots,\nu_d \in \Z^{n+1}$ the defining normals of the corresponding 
moment cone $C\subset\R^{n+1}$. By a result of Lerman~\cite{Le2}, the fundamental group $\pi_1(W_C)$ is isomorphic 
(canonically with respect to base points) to $\Z^{n+1}/\mathcal N$, where $\mathcal N$ is the $\Z$-span of $\{\nu_1, \ldots, 
\nu_d\}$. Any toric Reeb vector field $R_\nu$, $\nu\in\R^{n+1}$, has at least $m$ simple closed orbits $\gamma_1, \ldots, \gamma_m$, 
corresponding to the $m$ edges $E_1, \ldots, E_m$ of the cone $C$, and we will now determine how these simple closed 
orbits can be seen and related as elements of $\pi_1(W_C)$.

Consider an edge $E_\ell$ of $C$ obtained as the intersection of $n$ facets with normals $\nu_{\ell_1}, \ldots, \nu_{\ell_n}$. 
We can complete $\nu_{\ell_1}, \ldots, \nu_{\ell_n}$ to a $\Z^{n+1}$-basis $\nu_{\ell_1}, \ldots, \nu_{\ell_n}, \eta_\ell$. 
Given a toric Reeb vector field $R_\nu$, we can write  $\nu\in\R^{n+1}$ uniquely as
\[
\nu=\sum_{i=1}^n b_{\ell_i}\nu_{\ell_i}+b_\ell \eta_\ell\,.
\]
Then, the simple closed orbit $\gamma_\ell$ of $R_\nu$ corresponding to the edge $E_\ell$ represents the element 
$[\textup{sgn}(b_\ell)\eta_\ell]\in \Z^{n+1}/\mathcal N$ in the fundamental group of $W_C$. 
Consider the homomorphism $\phi: \Z\to \Z^{n+1}/\mathcal N$ defined by $\phi(1)=[\eta_\ell]$. Since 
\[
\eta_\ell\Z+\mathcal N\supseteq \eta_\ell\Z+\nu_{\ell_1}\Z+\ldots+\nu_{\ell_n}\Z=\Z^{n+1}\,,
\]
the homomorphism $\phi$ is surjective. Thus $\Z^{n+1}/\mathcal N\cong \Z/\ker \phi=\Z_{N_\ell}$ where 
\[
N_\ell=\min\{k\in \Z^+: k\eta_\ell\in \mathcal N\}\,.
\]
This shows that $\pi_1(W_C)$ is cyclic and is generated by any of the simple closed orbits $\gamma_\ell$ of $R_\nu$. 
Moreover it follows that $N_\ell=N$ is the same for each edge $E_\ell$ and is the order of the fundamental group of $W_C$. 
A way to compute this order $N$ is the following (see also~\cite{HLi}).

\begin{proposition}
The fundamental group of $W_C$ is a cyclic group of order $N$ where 
\[
N=\gcd\left\{\begin{vmatrix}
\vert &  & \vert\\ 
\nu_{i_1} & \ldots & \nu_{i_{n+1}}\\ 
\vert &  & \vert
\end{vmatrix}: 1\leq i_1<\ldots<i_{n+1}\leq d\right\}\,.
\]
In particular, if $C$ is determined by a toric diagram $D=\textup{conv}(v_1, \ldots, v_d)$ (i.e. $\nu_j=(v_j,1)$) then 
\[
N=\gcd\left\{\frac{1}{n!}\textup{Vol}(\textup{conv}(v_{i_1}, \ldots, v_{i_{n+1}})): 1\leq i_1<\ldots<i_{n+1}\leq d\right\}\,.
\]
\end{proposition}
\begin{proof}
Let 
\[
A=\begin{bmatrix}
\vert &  & \vert\\ 
\nu_{1} & \ldots & \nu_{d}\\ 
\vert &  & \vert
\end{bmatrix}\,,
\]
which is a $(n+1)\times d$ matrix. By the Smith normal form we have 
\[
\pi_1(W_C)\cong \ZZ^{n+1}/\mathcal N=\ZZ^{n+1}/\textup{im}(A)\cong \ZZ_{\alpha_1}\oplus \ldots \oplus \ZZ_{\alpha_{n+1}}\,,
\]
where $\alpha_1, \ldots, \alpha_{n+1}$ are the invariant factors of $A$ given by 
\[
\alpha_i=\frac{d_i(A)}{d_{i-1}(A)}\,,
\]
where $d_i(A)$ is the greatest common divisor of all the $i\times i$ minors of $A$ ($d_0(A)$ is set to be $1$). We claim that 
$d_n(A)=1$, from which follows that $\alpha_1=\ldots=\alpha_n = 1$. 

Indeed let $\nu_{\ell_1}, \ldots, \nu_{\ell_n}$ be the normals corresponding to facets intersecting at an edge $E_\ell$ of $C$ 
and let $\eta_\ell\in \ZZ^{n+1}$ be such that $\nu_{\ell_1}, \ldots, \nu_{\ell_n},\eta_\ell$ is a basis of $\ZZ^{n+1}$. Then
\[
\begin{vmatrix}\vert & & \vert & \vert\\ 
\nu_{\ell_1} & \ldots &\nu_{\ell_{n}} & \eta_\ell \\ 
\vert & & \vert & \vert
\end{vmatrix}=\pm 1\,.
\]
By the Laplace rule on the last column we get an integral combination of the $n\times n$ minors of 
\[
\begin{bmatrix}
\vert &  & \vert\\ 
\nu_{\ell_1} & \ldots & \nu_{\ell_n}\\ 
\vert &  & \vert
\end{bmatrix}
\]
giving $1$, hence their greatest common divisor is $1$, showing that $d_n(A)=1$.

The result now follows since $\pi_1(W_C)\cong \ZZ_{\alpha_{n+1}}$ and $\alpha_{n+1}=d_{n+1}(A)=N$. \qedhere
\end{proof}

If we now consider two different edges, say $E_\ell$ and  $E_k$, how can we relate the corresponding simple closed Reeb orbits,
$\gamma_\ell$ and $\gamma_k$, as elements in the fundamental group $\pi_1 (W_C)$? 
Let $\nu_{\ell_1}, \ldots, \nu_{\ell_n}, \eta_\ell$ and $\nu_{k_1}, \ldots, \nu_{k_n}, \eta_k$ be integral basis of $\ZZ^{n+1}$ corresponding to the edges $E_\ell$ and $E_k$, respectively. Consider the change of $\Z$-basis
\begin{align}
\eta_\ell&=\sum_{j=1}^nb_j\nu_{k_j}+b\eta_k  \label{eq1}\\
\nu_{\ell_i}&=\sum_{j=1}^na_{ij}\nu_{k_j}+a_i\eta_k  \label{eq2}
\end{align}
with $b_j, b, a_{ij}, a\in \ZZ$. Then
\[
\begin{vmatrix}
b_1 & \ldots & b_n &b \\ 
a_{11} & \ldots & a_{1n} &a_1 \\ 
\vdots &  & \vdots & \vdots\\ 
a_{n1} & \ldots & a_{nn} & a_n
\end{vmatrix}=\pm1\,.
\]
If follows from~(\ref{eq2}) that $a_i\eta_k\in \mathcal N$, and so we have that $N|a_i$ for $i=1, \ldots, n$. Thus, the determinant being $\pm 1$ gives
$$b\begin{vmatrix}
a_{11} & \ldots & a_{1n} \\
\vdots &  & \vdots \\ 
a_{n1} & \ldots & a_{nn} 
\end{vmatrix}\equiv \pm1 \mod N.$$
Equation~(\ref{eq1}) gives $[\eta_\ell]=b[\eta_k]$, that is, $[\eta_k]=c[\eta_\ell]$ where
\[
c\equiv b^{-1}\equiv \pm \begin{vmatrix}
a_{11} & \ldots & a_{1n} \\
\vdots &  & \vdots \\ 
a_{n1} & \ldots & a_{nn} 
\end{vmatrix} \mod N\,.
\]
Moreover, this is still valid if $a_{ij}$ satisfy the equations
\[
\nu_{\ell_i}\equiv \sum_{j=1}^na_{ij}\nu_{k_j} \mod N\,,
\]
which allows the computation of $c$ without determining $\eta_\ell, \eta_k$.

\subsection{Examples}
\label{ss:exanples}

\subsubsection{Prequantizations}

One source for examples of good toric contact manifolds is the prequantization construction over integral closed 
toric symplectic manifolds, i.e. $(M,\xi)$ with $M$ given by the $S^1$-bundle over $(B,\omega)$ with Euler class 
$-[\omega]/2\pi$ and $\xi$ being the horizontal distribution of a connection with curvature $\omega$. The 
corresponding good cones have the form
\[
C:= \left\{z(x,1)\in\R^{n}\times\R\mid x\in P\,,\ z\geq 0\right\}
\subset\R^{n+1}
\]
where $P\subset\R^n$ is a Delzant polytope with vertices in the integer lattice $\Z^n\subset\R^n$.
Note that if
\[
P = \bigcap_{j=1}^d \{x\in\R^{n}\mid \langle x, v_j \rangle + \lambda_j \geq 0\}\,,
\]
with integral $\lambda_1, \ldots, \lambda_d \in \Z$ and primitive $v_1,\ldots, v_d \in \Z^n$, then
the defining normals of $C \subset \R^{n+1}$ are
\[
\nu_j = (v_j, \lambda_j)\,,\ j=1, \ldots, d\,.
\]
When $\lambda_1 = \cdots \lambda_d = 1$ we have that $c_1(TB,\omega) = [\omega]/2\pi$ and the prequantization is a
Gorenstein toric contact manifold with toric diagram $D := \conv(v_1,\ldots, v_d) \subset \R^n$.

\subsubsection{$3$-dimensional lens spaces and their unit cosphere bundles}

A $3$-dimensional lens space $L^3_p(q):= L^3_p (1,q)$, where $p,q\in\N$ are coprime, is the quotient of 
$S^{3}\subset \C^{2}\setminus \{0\}$ by the free action of $\Z_p$ generated by 
\[
[1].(z_0, z_1)=\left(e^{\frac{2\pi i}{p}}z_0, e^{\frac{2\pi i q}{p}}z_1 \right)\,.
\]
Recall that the unit cosphere bundle $S^\ast N$ of any smooth (oriented) manifold $N$ has a canonical (co-oriented) 
contact structure with symplectization $T^\ast N\setminus \{\text{$0$-section}\}$ with its canonical symplectic structure.
We will now show that  any $S^\ast L^3_p(q)$ is a Gorenstein toric contact manifold with toric diagram given by a 
parallelogram in $\R^2$.

Let us start by characterizing the smooth $5$-manifolds that can be obtained from parallelograms in $\R^2$.
\begin{proposition} 
Let $D=\textup{conv}((0,0), (1,0), (q,p), (q+1, p))$ be a toric diagram where $p,q\in\N$ are coprime. Then the 
corresponding manifold $M=M_D$ is diffeomorphic to $S^\ast L\cong L \times S^2$ where $L=L^3_p(q)$.
\end{proposition}

\begin{remark}
Note that any parallelogram is $\textup{Aff}(2, \Z)$-equivalent to one of the diagrams considered in the theorem, 
so this shows that any toric diagram which is a parallelogram corresponds to $S^\ast L$ for some lens space $L$. 
The order $p$ of the fundamental group $p$ is given by the area of the toric diagram.
\end{remark}

\begin{proof}
The symplectization $W$ of $M$ is the toric symplectic cone determined by the good moment cone $C\subset\R^3$
with normals
\[
\nu_1=(q+1, p, 1)\,,\ \nu_2=(0, 0, 1)\,,\ \nu_3=(1, 0, 1)\quad\text{and}\quad \nu_4=(q, p, 1)\,.
\]
More precisely, it is the symplectic reduction of $(\C^4 \setminus\{0\}, \omega_{\rm st}, X_{\rm st})$ with respect
to the naturally induced action of $K\subset \T^4$, with $K := \ker \beta$ where $\beta: \T^4\to \T^3$ is represented 
by the matrix (cf.~(\ref{eq:defBeta}))
\[
\begin{bmatrix}
q+1 & 0 & 1 & q \\ 
p   & 0 & 0 & p \\ 
1   & 1 & 1 & 1
\end{bmatrix}.
\]
Hence we have that 
\[
K \cong \left\{\left(e^{it}, e^{it+\frac{2\pi i(q-1)k}{p}}, e^{-it-\frac{2\pi iqk}{p}}, e^{-it+\frac{2\pi i k}{p}}\right): t\in \R, k\in \Z\right\}
=S^1\times \langle \xi\rangle
\]
where $\xi=\left(1, e^{\frac{2\pi i(q-1)}{p}}, e^{-\frac{2\pi iq}{p}}, e^{\frac{2\pi i}{p}}\right)$ (notice that 
$\langle \xi \rangle\cong \Z_p$). This means that
\[
W = \{(z_1, z_2, z_3, z_4)\in \C^4 \setminus\{0\}: |z_1|^2+|z_2|^2 - |z_3|^2 - |z_4|^2= 0\} / (S^1\times \langle \xi \rangle)
\]
and $M=(S^3\times S^3)/(S^1\times \langle \xi \rangle)$ where  
\[
S^3\times S^3 = \{(z_1, z_2, z_3, z_4)\in \C^4: |z_1|^2+|z_2|^2=|z_3|^2+|z_4|^2=1\}\subset \C^4
\]
and 
\[
S^1=\{\left(e^{it}, e^{it}, e^{-it}, e^{-it}\right): t\in \R\}
\]
acts naturally on $S^3\times S^3 \subset \C^2 \times \C^2$.
By identifying $\C^2$ with the quaternions $\HH$ via $(z_1, z_2)\to z_1+jz_2$ we can also look at $S^3$ as the unit 
quaternions, giving $S^3$ a Lie group structure. Notice that the maps $(z_1, z_2, z_3, z_4)\to (z_1 : z_2) \in \C P^1\cong S^2$ 
and 
\[
(z_1, z_2, z_3, z_4)\to (z_1+jz_2)(z_3+z_4j)=(z_1z_3-\overline z_2 \overline z_4)+(z_1z_4+\overline z_2 \overline z_3)j
\in S^3\subseteq \HH
\]
are both invariant under the action of $S^1$. 
\begin{lemma}
The map $\psi: S^3\times S^3\to S^3 \times \C P^1 \cong S^3 \times S^2$ defined by
\[
(z_1, z_2, z_3, z_4)\to ((z_1z_3-\overline z_2 \overline z_4, z_1z_4+\overline z_2 \overline z_3), (z_1 : z_2))
\]
induces a diffeomorphism $(S^3\times S^3)/S^1\cong S^3 \times \C P^1$.
\end{lemma}
\begin{proof}
It suffices to show that $\psi$ is surjective with $S^1$-orbits as fibers.

\noindent \textit{Surjectivity.} 
Let $(x, z)\in S^3 \times \C P^1$. If we choose $z_1, z_2$ such that $|z_1|^2+|z_2|^2=1$ and $z_1/z_2=z$ and define 
$z_3, z_4$ by $z_3+z_4j=(z_1+jz_2)^{-1} x$ we get $\psi(z_1, z_2, z_3, z_4)=(x, z)$.

\noindent \textit{Fibers.} 
Suppose that $\psi(z)=\psi(z')$. From $z_1/z_2=z_1'/z_2'$ follows that there is $t\in \R$ such that $z_1'=e^{it}z_1$ and 
$z_2'=e^{it}z_2$. Moreover,
\[
\left(z_1+jz_2\right)\left(z_3+z_4j\right)=\left(z_1'+jz_2'\right)\left(z_3'+z_4'j\right)=\left(z_1+jz_2\right)\left(e^{it}z_3'+e^{it}z_4'j\right).
\]
Thus $z_3=e^{it}z_3'$ and $z_4=e^{it}z_4'$; that is, $(z_1', z_2', z_3', z_4')=(e^{it}z_1, e^{it}z_2, e^{-it}z_3, e^{-it}z_4)$, 
so $z$ and $z'$ are in the same $S^1$-orbit.\qedhere
\end{proof}
\noindent One can easily see that the $\langle\xi\rangle$-action on $S^3 \times S^3$ descends via $\psi$ to a 
$\langle\txi\rangle$-action on $S^3 \times \C P^1$ with
\[
\tilde \xi \cdot (z_1, z_2, z)=\left(e^{-\frac{2\pi iq}{p}}z_1, e^{\frac{2\pi i}{p}}z_2 , e^{-\frac{2\pi i(q-1)}{p}}z \right)
\]
for $(z_1, z_2)\in S^3$ and $z\in \C P^1$. 

Since $S^3$ is a Lie group, the unit sphere bundle $SS^3$ admits a trivialization $\varphi: SS^3\to S^3\times S^2$
and $S^3\times S^2$ can be identified with $S^3 \times \C P^1$ via $\textup{id}_{S^3}\times \textup{st}$, where 
$\textup{st}: S^2\to \C P^1$ is the standard identification between $S^2 \cong \R^2 \cup \{\infty\}$ and
$\C P^1 \cong \C \cup \{\infty\}$. We will show that, under these identifications, the 
$\langle\txi\rangle$-action on $S^3 \times \C P^1$ can be seen as a $\langle dg\rangle$-action on $SS^3$,
where $dg$ is the natural lift to $SS^3$ of the $\langle g \rangle$-action on $S^3$ determined by 
\[
g \cdot (z_1, z_2)=\left( e^{-\frac{2\pi iq}{p}}z_1, e^{\frac{2\pi i}{p}}z_2\right)\,.
\]
More precisely, we will prove that the diagram
\begin{center}
\begin{tikzcd}
S^3 \times \C P^1 \arrow[r, "\tilde \xi"]
&S^3 \times \C P^1\\
S^3\times S^2 \arrow[r, "g \times r_{2\alpha}"] \arrow[u, "\textup{id}\times \textup{st}"]
& S^3\times S^2 \arrow[u, "\textup{id}\times \textup{st}"] \\
SS^3 \arrow[r, "dg"] \arrow[u, "\varphi"]
& SS^3 \arrow[u, "\varphi"] \\
\end{tikzcd}
\end{center}
commutes, where $r_{2\alpha}$ is a certain rotation of $S^2$ with angle $2\alpha$ (see below).
It will then follow that
\[
M=(S^3\times S^2)/\langle\tilde \xi\rangle\cong SS^3/\langle dg\rangle\cong S\left(S^3/\langle g \rangle\right)\cong 
S L_p^3(q)\cong L_p^3(q) \times S^2
\]
as desired.

To show that the diagram commutes,  consider again $S^3\subset \HH$ as the unit quaternions. Since
\[
e^{-\frac{2\pi iq}{p}}z_1+e^{\frac{2\pi i}{p}}z_2j=e^{-\frac{\pi i(q-1)}{p}}(z_1+z_2j)e^{-\frac{\pi i(q+1)}{p}}
\]
the $\langle g \rangle$-action on $S^3$ can be written as
\[
g\cdot x=e^{i\alpha} x e^{i\beta}\,, \ x\in S^3\subset \HH\,,
\]
where $\alpha=-\frac{\pi(q-1)}{p}$, $\beta=-\frac{\pi(q+1)}{p}$. We can now think of the tangent space 
$S_x S^3$ at $x$ as the set of unit quaternions $v\in \HH$ perpendicular to $x$. Thus $SS^3$ consists 
of pairs of orthogonal quaternions $(x,v)\in \HH\times \HH$ such that $|x|=|v|=1$. We have a trivialization 
\begin{equation} \label{eq:varphi}
\varphi: SS^3\to S^3 \times S_e S^3\cong S^3\times S^2 
\quad\text{given by} \quad \varphi(x,v)=(x, vx^{-1})\,.
\end{equation} 
Here $S_e S^3$ is the unit sphere bundle at the identity of $S^3$, that is,
\[
S_e S^3=\{ai+bj+ck: a^2+b^2+c^2=1\}\cong S^2.
\]
For $(x,v)\in SS^3$ we have that
\[
\varphi ((dg)\cdot (x,v)) = \varphi \left( e^{i\alpha} x e^{i\beta}, e^{i\alpha} v e^{i\beta} \right) =
\left( e^{i\alpha}xe^{i\beta} , e^{i\alpha}vx^{-1}e^{-i\alpha}\right) 
= \left( g\cdot x,  r_{2\alpha}(vx^{-1})\right)\,,
\]
where $r_{2\alpha}(w)=e^{i\alpha}we^{-i\alpha}$ is the rotation of $S^2=S_eS^3$ with angle $2\alpha$ 
about the $i$-axis. This shows that the down square of the diagram commutes. 

Consider now the $\langle\txi\rangle$-action on $S^3\times S^2$. It is clear that on the first coordinate 
$\txi$ acts as $g$. If we give spherical coordinates $(\theta, \phi)$ to $S^2$ with respect to the $i$-axis, 
then the stereographic projection $\textup{st}: S^2\to \C P^1$ from $i$ is given by 
$\textup{st}(\theta, \phi)=\cot(\theta/2)e^{i\phi}$ and the rotation around the $i$-axis with angle $2\alpha$ is given by 
$r_{2\alpha}(\theta, \phi)=(\theta, \phi+2\alpha)$. Therefore
\[
e^{2i\alpha} \cdot \textup{st}(\theta, \phi)=\cot(\theta/2)e^{i(\phi+2\alpha)}= 
\textup{st}(\theta, \phi+2\alpha)=\textup{st}(r_{2\alpha}(\theta, \phi)).
\]
From this we can conclude that the top square commutes.\qedhere
\end{proof}

We will now prove that the toric contact structures on $L^3_p (q) \times S^2$ are indeed the standard contact 
structures on the unit cosphere bundles $S^\ast L^3_p (q)$. Let us start with the case $p=1$ and $q=0$, i.e. when
the toric diagram is 
$$D=\textup{conv}((0,0), (1,0), (0,1), (1, 1))\,.$$
The corresponding moment cone in $\R^3$ has
normals
\[
\nu_1=(1, 1, 1)\,,\ \nu_2=(0, 0, 1)\,,\ \nu_3=(1, 0, 1)\quad\text{and}\quad \nu_4=(0, 1, 1)\,.
\]
Via the $SL(3, \Z)$ transformation 
\[
\begin{bmatrix}
1 & 1 & -1 \\ 
-1   & 1 & 0 \\ 
0  & -1 & 1 \end{bmatrix}\,,
\]
this cone is equivalent to the one with normals 
\[
\tnu_1=(1, 0, 0)\,,\ \tnu_2=(-1, 0, 1)\,,\ \tnu_3=(0, -1, 1)\quad\text{and}\quad \tnu_4=(0, 1, 0)\,,
\]
which is the cone over the Delzant polytope $[0,1]\times [0,1]$. This means that $(M_D, \xi_D)$ is the prequantization
of $(B=S^2\times S^2, \omega = \sigma_1 + \sigma_2)$, where $\sigma_i = \pi_i^\ast (\sigma)$ with
$\pi_i : S^2 \times S^2 \to S^2$, $i=1,2$, the factor projections and $\int_{S^2} \sigma = 2\pi$.
This prequantization is well known to be contactomorphic to $S^\ast S^3$ with its canonical contact structure, but
we present here a proof for completeness.
\begin{proposition}
The prequantization of $(S^2\times S^2, \sigma_1+\sigma_2)$ is $S^\ast S^3$ with its canonical 
contact structure.
\end{proposition}
\begin{proof}
We can use the bi-invariant metric $\langle \cdot,\cdot\rangle$ on $S^3$ to identify $S^\ast S^3$ with $SS^3$. This
provides $SS^3$ with a contact form $\alpha$ which is given by $\alpha_{(x,v)}(U)=\langle pr_\ast U,v\rangle_x$ for 
$(x,v)\in SS^3$ and $U\in T_{(x,v)}SS^3$ where $pr : SS^3\to S^3$ is the canonical projection. 
The Reeb flow for this contact form is the geodesic flow on $SS^3$.

Consider the diffeomorphism $\varphi: S S^3\to S^3\times S^2$ defined by~(\ref{eq:varphi}). Since the usual 
Riemannian metric of $S^3$ is bi-invariant for the Lie group structure on $S^3$, under this diffeomorphism the 
geodesic flow is given by $(t, (x, w))\mapsto (x\exp(tw),w)$; note that $t\mapsto \exp(tw)$ has period $2\pi$ 
for every $w\in S^2$, so the geodesic flow gives a $S^1$-action on $SS^3$. We can check that this action gives a 
principal $S^1$-bundle $\pi: SS^3\to S^2\times S^2$ where $\pi$ is defined by $\pi(x,v)=(vx^{-1}, x^{-1}v)$. 

We want to show that the curvature form $\omega\in \Omega^2(S^2\times S^2)$ of this bundle for the connection 
determined by the contact form $\alpha$ is $\sigma_1 + \sigma_2$; recall that $\omega$ is determined by the 
formula $d\alpha=\pi^\ast \omega$. 

We begin by showing that $\omega$ is invariant for the $SO(3)\times SO(3)$-action on $S^2\times S^2$. 
First, we note that $S^3\times S^3$ acts on $SS^3$ via left and right multiplication, that is,
\[
(g_1,g_2)\cdot (x,v)=(dL_{g_1})(dR_{g_2})^{-1}(x,v)=( g_1xg_2^{-1},  g_1vg_2^{-1} )\,.
\]
Moreover, it acts on $S^2\times S^2$ via conjugation on both coordinates:
\[
(g_1, g_2)\cdot (w_1, w_2)=(g_1w_1g_1^{-1}, g_2w_2g_2^{-1})\,.
\]
We can easily check that $\pi: SS^3\to S^2\times S^2$ is equivariant with respect to these actions. Thus, if we show 
that $\alpha$ is invariant with respect to the $S^3\times S^3$ action on $SS^3$ it follows that $\omega$ is invariant 
with respect to the $S^3\times S^3$ action on $S^2\times S^2$ which descends to the usual $SO(3)\times SO(3)$-action 
by rotations on each component. Indeed, we have:
\begin{align*}
[(dL_{g_1})^\ast \alpha]_{(x,v)}(U) & =\alpha_{(g_1x,g_1v)}((dL_{g_1})_\ast U) =
\langle g_1v, pr_\ast(dL_{g_1})_\ast U\rangle_{g_1x} \\
& = \langle g_1v, g_1 pr_\ast U \rangle_{g_1x} = \langle v, pr_\ast U\rangle_x =\alpha_{(x,v)}(U)\,.
\end{align*}
Hence $(dL_{g_1})^\ast \alpha=\alpha$. Note that we used that $pr\circ dL_{g_1}=L_{g_1}\circ pr$ and that the metric 
$\langle\cdot,\cdot\rangle$ is left-invariant. Similarly $(dR_{g_1})^\ast \alpha=\alpha$, so the claim follows. 

Now, since the $SO(3)\times SO(3)$-action is transitive, it is clear that this action preserves $\omega$ if and only if 
$\omega$ has the form $\lambda_1 \sigma_1 + \lambda_2 \sigma_2$ for some $\lambda_1, \lambda_2\in \R$. 
We compute $\lambda_1$ by computing the integral of $\omega$ on the sphere $S^2\times \{i\}$. We can check that 
$\pi$ maps the disk
\[
D=\{(x, xi)\in SS^3: x=a+bj+ck\textup{ with } a^2+b^2+c^2=1, a\geq 0\}\subseteq SS^3
\]
onto $S^2\times \{i\}$ and is injective in $D\setminus \partial D$ (and collapses the boundary to the point $(-i, i)$). 
Hence
\[
\lambda_1= \frac{1}{2\pi} \int_{S^2\times \{i\}}\omega= \frac{1}{2\pi} \int_{D} \pi^\ast \omega= \frac{1}{2\pi} \int_{\partial D} \alpha
\]
by Stokes' theorem and $\pi^\ast \omega=d\alpha$. But $\partial D$ is the closed orbit of the Reeb flow (since it is 
mapped by $\pi$ to a point) and the closed orbits of the Reeb flow have period $2\pi$, hence 
$\lambda_1= \frac{1}{2\pi} \int_{\partial D}\alpha=1$. 
Similarly $\lambda_2=1$ and thus $\omega= \sigma_1 + \sigma_2$.
\qedhere
\end{proof}
Now we prove that in general the contact structure on $M_D\cong SL$ is the canonical one. The inclusion 
$\ker \beta_D\subseteq \ker \beta_Q$ induces a projection 
$$S^3\times S^2\cong M_Q\to M_D\cong S^3\times S^2/\langle \widetilde \xi\rangle$$ 
which by naturality of the Delzant construction is a contact transformation (and it is also a local diffeomorphism). 
This projection fits in the following commutative diagram:
\begin{center}
\begin{tikzcd}
SS^3\arrow[d, "\varphi"] \arrow[r, "d \overline{p}"]&
S L^3_p(q)  \arrow[d, "\overline{\varphi}"]\\
M_Q \arrow[r]&
M_D
\end{tikzcd}
\end{center}
where $\overline{p}: S^3\to L^3_p (q)$ is the quotient projection. We already showed that $\varphi$ is a contactomorphism, 
$d \overline{p}$ is also a contact transformation and a local diffeomorphism. Hence it follows that $\overline \varphi$ 
is a contactomorphism. That is, we just showed the following:

\begin{theorem} \label{thm:lensDiag}
The contact manifold $S^\ast L^3_p (q)$ is a Gorenstein toric contact manifold with corresponding toric 
diagram
\[
\textup{conv}((0,0), (1,0), (q,p), (q+1, p)).
\]
\end{theorem}

Moreover, we can understand the action of $T^3$ on $S^\ast L^3_p (q)$. For $L^3_1 (0) =S^3$, this action is the action of 
$T^3=T^2\times S^1$ where $S^1$ acts on $S^\ast S^3$ by the geodesic flow and the action of $T^2$ is 
the lift of the usual action of $T^2$ on $S^2\times S^2$ via $\pi: S^\ast S^3\to S^2\times S^2$. But the later is 
easily seen to be induced by the action of $T^2$ on $S^3$ (the same giving $S^3$ a toric structure). Explicitly, 
identifying $S^\ast S^3$ with $S^3 \times S^2$ as before, the action is given by
\[
\left(e^{i\alpha}, e^{i\beta}, e^{i\gamma}\right)\cdot (x,w) = 
\left( e^{i\beta}e^{\gamma w} x e^{i\alpha} , e^{i\beta}we^{-i\beta} \right).
\]

In general, we have an action of $\widetilde T^2\times S^1$ on $S^\ast L^3_p (q)$ where $\widetilde T^2=T^2/(\Z_p)$ 
still comes from the usual toric structure on $L^3_p (q)$ and $S^1$ still is the geodesic flow. Note that when $p>1$ 
the geodesic flow action is no longer free.

\begin{remark} \label{rmk:notToric}
The contact manifold $S^\ast S^{n+1}$, or more generally $S^\ast L^{n+1}$, is not toric when $n>2$. 
In fact, $H^2(S^\ast S^{n+1}, \R)=0$ for $n>2$ and, by~\cite{Le2}, this implies that if $S^\ast S^{n+1}$ was toric
then its moment cone $C\subset\R^{n+1}$ would have $d=n+1$ facets. Any such moment cone gives rise to a toric
contact manifold diffeomorphic to $S^{2n+1}/\ker \beta$, where $\ker \beta$ is a finite group (cf.~(\ref{eq:defBeta})).
Since $S^\ast S^{n+1}$ is simply connected, this would be possible only if $S^\ast S^{n+1}$ was diffeomorphic to 
$S^{2n+1}$, which clearly is not the case since they have different cohomology types. 
\end{remark}

\subsubsection{Higher dimensional lens spaces} \label{ss:higherlens}

We will now consider higher dimensional lens spaces, i.e. quotients of odd dimensional spheres by cyclic groups. 
More concretely, regarding $S^{2n+1}$ as a subset of $\C^{n+1}\setminus \{0\}$ we have an action of $\Z_p$ 
generated by 
\[
[1].(z_0, \ldots, z_n)=\left(e^{\frac{2\pi i \ell_0}{p}}z_0, e^{\frac{2\pi i \ell_1}{p}}z_1, \ldots, e^{\frac{2\pi i \ell_n}{p}}z_n \right)
\]
where $\ell_0, \ldots, \ell_n$ are integers called the weights of the action. Such an action is free when the weights 
are coprime with $p$ and in that case we have a lens space obtained as the quotient of $S^{2n+1}$ by the action of 
$\Z_p$. We denote such a lens space by
\[
L_p^{2n+1}(\ell_0, \ell_1, \ldots, \ell_n)\,.
\]
Note that different sequences of weights might produce diffeomorphic lens spaces, for instance by permuting the weights, 
multiplying every weight by some $k$ coprime with $p$ and changing the sign of some weights. Moreover these are the 
only possibilities leading to diffeomorphic lens spaces (see section 12 of~\cite{M}).

Suppose now that $L^{2n+1}$ is a lens space admitting a toric contact structure. We have that $H^2(L, \R)=0$,
since $H^2(L; \Z)\cong \Z_p$, which again implies that the moment cone of $L$ has $d=n+1$ facets (see Remark~\ref{rmk:notToric}). 
Restricting ourselves
to Gorenstein lens spaces, Proposition~\ref{prop:c_1} implies that the moment cone of the symplectization of $L$ must be
$SL(n+1,\Z)$ equivalent to a cone $C\subset\R^{n+1}$ with defining normals
\[
\nu_0=\begin{bmatrix}
0\\ 
0\\
\vdots\\ 
0\\ 
0\\ 
1
\end{bmatrix},\,\,
 \nu_j=\begin{bmatrix}
0\\ 
\vdots\\ 
1\\ 
0\\
\vdots\\
1
\end{bmatrix}\textup{ for }j=1, \ldots, n-1,\,\, \nu_n=\begin{bmatrix}
\alpha_1\\ 
\alpha_2\\
\vdots\\ 
\alpha_{n-1}\\ 
p\\ 
1
\end{bmatrix}\,,\ \alpha_1,\ldots,\alpha_{n-1} \in\Z\,,\ p\in\N\,.
\]
The system $\beta(x)=0$, with $\beta$ given by~(\ref{eq:defBeta}), can be written as
\[
\begin{cases}
x_1+\alpha_1 x_{n}=0\\ 
\,\,\,\,\,\,\,\,\,\vdots\\ 
x_{n-1}+\alpha_{n-1} x_{n}=0\\ 
px_{n}=0\\ 
\sum_{i=0}^{n}x_i=0
\end{cases}
\]
which has $p$ solutions in $\T^{n+1}$, generated by 
\[
\left(\frac{\alpha_0}{p}, -\frac{\alpha_1}{p}, \ldots, -\frac{\alpha_{n-1}}{p}, \frac{1}{p}\right)\in \T^{n+1}\,,
\text{with}\ \alpha_0:=\alpha_1+\ldots+\alpha_{n-1}-1\,.
\]
Hence the corresponding contact manifold is 
\[
L_p^{2n+1}(\alpha_0, -\alpha_1, \ldots, -\alpha_{n-1}, 1)\,.
\]
It is not hard to see that $C$ is a good cone if and only if $\gcd (\alpha_j, p)=1$ for $j=0,1, \ldots, n-1$.
Note that when $n$ is even this implies in particular that $p$ must be odd.

\section{Conley-Zehnder index of non-contractible closed orbits}
\label{s:trivialization}

Let $(M^{2n+1},\xi)$ be a good toric contact manifold and $(W^{2(n+1)}, \omega, X)$ its associated toric
symplectic cone, obtained via symplectic reduction of $(\C^d \setminus\{0\} \cong \R^{2d} \setminus\{0\},
\omega_{\rm st}, X_{\rm st})$ by the linear action of a subtorus $K\subset\T^d$ (see~(\ref{eq:defK})).
Denote by $F: \C^d \setminus\{0\} \to \fk^\ast$ the associated moment map, so that $W = Z/K$ with
$Z:= F^{-1} (0)$.

Let $\alpha$ be a toric contact form supporting $\xi$. As discussed in \cite[section 3]{AM2} (c.f. \cite{Esp,McL}) 
and assuming that $c_1(\xi)=0$, the index of every closed orbit of $\alpha$ can be defined using a trivialization 
of the determinant line bundle $\Lambda_\C^n\xi$. In this section, we will show how to compute this index for 
non-contractible orbits using the above symplectic reduction description. 

In order to do this, let us first explain how we dealt with contractible orbits in \cite{AM1}. Consider a periodic orbit 
$\ga$ of $\alpha$. If $\ga$ is contractible then it admits a lift to a closed orbit $\hga$ in $Z$ of a suitable linear 
unitary Hamiltonian flow on $\C^d$ and \cite[Lemma 3.14]{AM1} establishes that the Robbin-Salamon index of 
$\hga$ with respect to the usual (constant) symplectic trivialization of 
$T\C^d$ equals the Robbin-Salamon index of $\ga$ with respect to the trivialization induced by a capping disk. 
(The Robbin-Salamon index \cite{RS} coincides with the Conley-Zehnder whenever the periodic orbit is non-degenerate.)

Since $Z$ is simply connected (c.f. \cite[section 3]{AM1}), this lift exists if, and only if, $\ga$ is contractible. 
Thus, if $\ga$ is not contractible we do not have a closed lift  and need to consider an arc of trajectory in $\C^d$. 
More precisely, let $m$ be the smallest positive integer such that $\ga^m$ is contractible (recall that the 
fundamental group of $M$ is finite cyclic) and let $\hga$ be a closed lift of $\ga^m$ as before. Suppose, 
without loss of generality, that $\ga^m$ has period one so that $\ga$ has period $1/m$. Consider the arc 
$\bga:=\hga|_{[0,1/m]}$. We can compute the index of $\bga$ using the constant trivialization of $T\C^d$ 
but it turns out that, in general, this index is not equal to the index of $\ga$ if $m>1$. We need to ``correct'' 
the symplectic path along $\bga$ as stated in the next proposition.

Before we give the precise statement, let us introduce some terminology. Let $g$ be an element of $K$ 
such that 
\begin{equation} \label{def:g}
\bga(1/m)=g\bga(0)\quad\text{and}\quad g^m=\id\,. 
\end{equation}
Since $c_1(\xi)=0$ every element of $K$ is in $\SU(d)$.
Let $\psi: [0,1/m] \to \SU(d)$ be a smooth map such that $\psi(0)=\id$ and $\psi(1/m)=g$. Denote by 
$\vr^H_t$ the Hamiltonian flow on $\C^d$ for which $\bga$ is an arc of trajectory. As mentioned before, 
$\vr^H_t$ is a linear unitary flow.

\begin{proposition} \label{prop:index}
The Robbin-Salamon index of $\psi(t)^{-1}\circ \vr^H_t$ with respect to the constant symplectic trivialization 
of $T\C^d$ is equal to the Robbin-Salamon index of $\ga$ defined using a trivialization of the determinant 
line bundle $\Lambda_\C^n\xi$. 
\end{proposition}

\begin{proof}
Let $\pi: Z \to W$ be the quotient projection. 
Choose a basis of the dual Lie algebra $\fk^\ast$ of $K$ and let $F_1,\dots,F_{d-n-1}$ be the corresponding moment 
map components. Let $\D$ be the distribution on $Z$ given by
\[
\D = \text{span}\{X_{F_1},\dots,X_{F_{d-n-1}},\nabla F_1,\dots,\nabla F_{d-n-1}\},
\]
where the gradient is taken with respect to the Euclidean metric. Clearly $\D$ is a symplectic distribution 
and we can choose the basis of $\fk^\ast$ such that 
$$\{X_{F_1},\dots,X_{F_{d-n-1}},\nabla F_1,\dots,\nabla F_{d-n-1}\}$$
defines a symplectic basis of $\D$. Denote by $\D^\om$ the symplectic orthogonal of $\D$ and note that 
$d\pi|_{\D^\om}: \D^\om \to TW$ is a symplectic isomorphism at every point of $Z$.

Choose a symplectic frame $(u_1,\dots,u_{2n+2})$ of $\ga^*TW$ obtained via a global trivialization of 
$\Lambda_\C^n\xi$ (c.f. \cite[section 3]{AM2}).
Let $(\hat u_i)=(d\pi|_{\D^\om})^{-1}(u_i)$ be the frame of $\bga^*\D^\om$ given 
by the lift of $(u_i)$ and define the (ordered) frame $(v_1,\dots,v_{2d})$ of $\bga^*T\C^d$ as 
\[
(\hat u_1,\dots,\hat u_{2n+2},X_{F_1},\dots,X_{F_{d-n-1}},\nabla F_1,\dots,\nabla F_{d-n-1}).
\]
Denote by $\Phi$ the trivialization of $\bga^*T\C^d$ induced by $(v_i)$. Since $d\vr^H_t(X_{F_i})=X_{F_i}$ and 
$d\vr^H_t(\nabla F_i)=\nabla F_i$ for every $i \in \{1,\dots,d-n-1\}$ and $t\in \R$, we have, 
by construction, that
\[
\rs(\ga)=\rs(\bga;\Phi)
\]
where the index on the left hand side is computed using the frame $(u_i)$.

Now, consider the symplectic frame $(w_i)$ of $\bga^*T\C^d$ given by $w_i(t)=\psi(t)(v_i(0))$ and let 
$\Psi$ be the trivialization of $\bga^*T\C^d$ induced by $(w_i)$. It is easy to see that
\[
\rs(\bga;\Psi)=\rs(\psi(t)^{-1}\circ \vr^H_t),
\]
where the index on the right hand side is computed using the constant trivialization of $T\C^d$. 
Therefore, we have to show that
\begin{equation}
\label{eq:indexes}
\rs(\bga;\Phi) = \rs(\bga;\Psi).
\end{equation}
To accomplish this equality, consider the extensions of the frames $(v_i)$ and $(w_i)$ to 
$\hga^*T\C^d$ given by
\[
v'_i(t+j/m)=dg^j(v_i(t))\ \text{and}\ w'_i(t+j/m)=dg^j(w_i(t))
\]
for every $t \in [0,1/m]$ and $j \in \{0,\dots,m-1\}$. Since $g^j\circ\pi=\pi$, $dg^j(X_{F_i})=X_{F_i}$ and 
$dg^j(\nabla F_i)=\nabla F_i$, we have that
\begin{equation}
\label{eq:triv v'}
(v'_1,\dots,v'_{2n+2},v'_{2n+3},\dots,v'_{2d})
\end{equation}
\[
=(\hat u'_1,\dots,\hat u'_{2n+2},X_{F_1},\dots,X_{F_{d-n-1}},\nabla F_1,\dots,\nabla F_{d-n-1}),
\]
where $(\hat u'_i)=(d\pi|_{\D^\om})^{-1}(u'_i)$ is the lift of the obvious extension $(u'_i)$ of the frame 
$(u_i)$ to $(\ga^m)^*TW$. It follows from this that
\begin{equation}
\label{eq:index Phi'}
\rs(\hga;\Phi')=\rs(\ga^m),
\end{equation}
where $\Phi'$ is the trivialization of $\hga^*T\C^d$ induced by $(v'_i)$ and the index on the right hand side 
is computed using the frame $(u'_i)$. We also have that
\begin{equation}
\label{eq:triv w'}
w'_i(t)=\psi'(t)(v_i(0))
\end{equation}
where $\psi': [0,1] \to \SU(d)$ is the extension of $\psi$ given by $\psi'(t+j/m)=g^j\circ\psi(t)$ for every 
$t \in [0,1/m]$ and $j \in \{0,\dots,m-1\}$. Note that $\psi'$ is a loop in $\SU(d)$ based at the identity 
and therefore,
\begin{equation}
\label{eq:index Psi'}
\rs(\hga;\Psi')=\rs(\hga)
\end{equation}
where $\Psi'$ is the trivialization of $\hga^*T\C^d$ induced by $(w'_i)$ and the index on the right hand side 
is computed using the constant trivialization of $T\C^d$. We have from \cite[Lemma 3.4]{AM1} that the 
right hand sides of \eqref{eq:index Phi'} and \eqref{eq:index Psi'} are equal. Hence, we arrive at
\begin{equation}
\label{eq:index Phi'=index Psi'}
\rs(\hga;\Phi') = \rs(\hga;\Psi').
\end{equation}

Consider the map $A': [0,1] \to \Sp(2d)$ uniquely defined by the property that $A'_t v'_i(t)=w'_i(t)$ for every 
$i$ and $t$. Using \eqref{eq:triv v'}, \eqref{eq:triv w'} and the fact that $g^j\circ\pi=\pi$, $dg^j(X_{F_i})=X_{F_i}$ 
and $dg^j(\nabla F_i)=\nabla F_i$ we conclude that
\[
v'_i(j/m)=w'_i(j/m)\ \text{for every}\ i \implies A'_{j/m}=\id
\]
for every $j \in \{0,\dots,m\}$. We have that
\begin{equation}
\label{eq:indexes hga}
\rs(\hga;\Phi') = \rs(\hga;\Psi') + 2\maslov(A')
\end{equation}
and
\begin{equation}
\label{eq:indexes bga}
\rs(\bga;\Phi) = \rs(\bga;\Psi) + 2\maslov(A)
\end{equation}
where $A=A'|_{[0,1/m]}$. We claim that
\begin{equation}
\label{eq:indexes A and A'}
\maslov(A')=m\maslov(A).
\end{equation}
Indeed,
\begin{align*}
g^jA_tv'_i(t) & = g^jw'_i(t) \\
& = w'_i(t+j/m) \\
& = A'_{t+j/m}v'_i(t+j/m) \\
& = A'_{t+j/m}g^jv'_i(t) \\
\end{align*}
for every $i$ and consequently
\begin{equation}
\label{eq:loops}
A'_{t+j/m} = g^jA_tg^{-j}
\end{equation}
for every $t \in [0,1/m]$ and $j \in \{0,\dots,m-1\}$. But the conjugation map $P \mapsto g^jPg^{-j}$ is clearly 
homotopic to the identity and therefore induces the identity map on $\pi_1(\Sp(2d))$. Thus, all the loops 
$t \in [0,1/m] \mapsto A'_{t+j/m}$, $j \in \{0,\dots,m-1\}$, are homotopic to the loop $t \in [0,1/m] \mapsto A_t$. 
Since $A'$ is a concatenation of these loops, we conclude the equality \eqref{eq:indexes A and A'}.

Finally, we deduce from \eqref{eq:index Phi'=index Psi'} and \eqref{eq:indexes hga} that $\maslov(A')=0$. 
Thus, \eqref{eq:indexes} follows immediately from \eqref{eq:indexes bga} and  \eqref{eq:indexes A and A'}, 
proving the proposition.
\end{proof}

As we will now explain, Proposition~\ref{prop:index} implies in particular that the explicit method to compute the
Conley-Zehnder index of any contractible non-degenerate closed toric Reeb orbit described in~\cite[section 5]{AM1}
also applies to non-contractible orbits.

Given a toric diagram $D = \conv (v_1, \ldots, v_d) \subset \R^n$ and corresponding Gorenstein toric contact manifold 
$(M_D, \xi_D)$, consider a toric Reeb vector field $R_\nu \in \Xx (M_D, \xi_D)$ determined by the normalized toric 
Reeb vector (cf. Proposition~\ref{prop:reeb})
\[
\nu = (v,1) \quad\text{with}\quad v = \sum_{j=1}^d a_j v_j\,,\ a_j\in\R^+\,,\ j=1, \ldots, d\,,\ \text{and}\ 
\sum_{j=1}^d a_j = 1\,.
\]
By a small abuse of notation, we will also write
\[
R_\nu = \sum_{j=1}^d a_j \nu_j \,,
\]
where $\nu_j = (v_j,1)$, $j=1,\ldots,n$, are the defining normals of the associated good moment
cone $C\subset\R^n$.
Making a small perturbation of $\nu$ if necessary, we can assume that
\[
\text{the $1$-parameter subgroup generated by $R_\nu$ is dense in $\T^{n+1}$,}
\]
which means that if $v = (r_1, \ldots,  r_n)$ then $1,r_1,\ldots,r_n$'s are $\Q$-independent.
This is equivalent to the corresponding toric contact form being non-degenerate. 
In fact, the toric Reeb flow of $R_\nu$ on $(M_D,\xi_D)$ has exactly $m$ simple closed
orbits $\gamma_1, \ldots,\gamma_m$, all non-degenerate, corresponding to the $m$ edges
$E_1,\ldots,E_m$ of the cone $C$, i.e. one non-degenerate closed simple toric $R_\nu$-orbit
for each $S^1$-orbit of the $\T^{n+1}$-action on $(M_D,\xi_D)$. Equivalently, there is
\emph{one non-degenerate closed simple toric $R_\nu$-orbit for each facet of the toric diagram $D$.}
Let $\gamma$ denote one of those non-degenerate closed simple toric $R_\nu$-orbits and assume 
without loss of generality that the vertices of the corresponding facet, necessarily a simplex, are 
$v_1, \ldots, v_n \in \Z^n$. Let $h\in\Z^n$ be such that
\[
\{ \nu_1 = (v_1,1), \ldots, \nu_n = (v_n,1), \eta = (h, 1)\} \ \text{is a $\Z$-basis of $\Z^{n+1}$.}
\]
Then $R_\nu$ can be uniquely written as
\[
R_\nu = \sum_{j=1}^n b_j \nu_j + b \eta\,,\ \text{with}\ b_1,\ldots, b_n\in \R\ \text{and}\ 
b = 1 - \sum_{j=1}^n b_j \ne 0\,,
\]
and for contractible $\gamma$, as shown in~\cite[section 5]{AM1}, the Conley-Zehnder index of $\gamma^{N}$, 
for any $N\in\N$, is given by
\begin{equation} \label{eq:CZindex0}
\mu_{\rm CZ} (\gamma^{N}) = 2 \left( \sum_{j=1}^n \left\lfloor N \frac{b_j}{|b|}\right\rfloor
 + N \frac{b}{|b|} \sum_{j=1}^d \teta_j \right) + n\,,
\end{equation}
where $\teta\in\Z^d$ is an integral lift of $\eta\in\Z^{n+1}$ under the map $\beta: \Z^{d} \to \Z^{n+1}$ defined
by~(\ref{eq:defBeta}). The value in the last coordinate of $\nu_1 = (v_1,1), \dots, \nu_d = (v_d,1)$ and $\eta = (h,1)$ implies that 
$\sum_{j} \teta_j  = 1$ and so
\begin{equation} \label{eq:CZindex}
\mu_{\rm CZ} (\gamma^{N}) = 2 \left( \sum_{j=1}^n \left\lfloor N \frac{b_j}{|b|}\right\rfloor
 + N \frac{b}{|b|} \right) + n\,.
\end{equation}
When $\gamma$ is not contractible $(\eta,1)\in\Z^{n+1}$ does not have an integral lift under the map $\beta$, but it does
have an integral lift modulo $\frac{1}{2\pi}\beta(K)\subset \Z^{n+1}$. In other words, there is $\teta\in\Z^d$ such that
\[
\beta (\teta) = \eta - \frac{\beta (g)}{2\pi}\,,\ \text{where $g\in K\cap SU(d)$ is characterized by~(\ref{def:g}).}
\]
It follows from Proposition~\ref{prop:index} that the index of $\gamma^N$ is then given by~(\ref{eq:CZindex0}) for this $\teta$.
Since $g\in SU(d)$ we still have that $\sum_{j} \teta_j  = 1$ and so formula~(\ref{eq:CZindex}) remains valid when $\gamma$ 
is not contractible.

\section{Contact Betti numbers of $S^\ast (L^3_p (q))$}
\label{s:examples1}

In this section we will use formula~(\ref{eq:CZindex}) for the Conley-Zehnder index to compute the contact Betti numbers of the 
unit cosphere bundle of $3$-dimensional lens spaces  $(L^3_p (q))$, which by Theorem~\ref{thm:lensDiag} are the Gorenstein
toric contact $5$-manifolds determined by toric diagrams of the form
\[
\textup{conv}(v_0 = (0,0), v_1=(1,0), v_2 = (q,p), v_3= (q+1,p))\subset\R^2 \quad
\textup{with $p,q\in\N$ coprime.}
\]

Let $E_1, E_2, E_3, E_4$ be the edges of the toric diagram with endpoints $\{v_0, v_1\}$, $\{v_0, v_2\}$, $\{v_1, v_3\}$, $\{v_2, v_3\}$.
Consider a normalized toric Reeb vector
$R_\varepsilon=(\varepsilon_1, \varepsilon_2, 1)$, with $0<\varepsilon_1, \varepsilon_2<<1$, and denote by $\gamma_j$ its simple 
closed Reeb orbit corresponding to $E_j$, $j=1,2,3,4$. By taking $\varepsilon_1,\varepsilon_2$ such that $1, \varepsilon_1,\varepsilon_2$
are $\Q$-independent, 
we know that the Reeb flow of $R_\varepsilon$ has exactly these four simple closed Reeb orbits and they are all non-degenerate.
Either by direct computation using~(\ref{eq:CZindex}) or by using~\cite[Theorem 1.6]{AM2}, we also know that by considering
$\varepsilon_1$ and $\varepsilon_2$ arbitrarily small we will have that $\mu_{CZ}(\gamma_1^N)$ and 
$\mu_{CZ}(\gamma_2^N)$ are arbitrarily large for any $N\in\N$. Hence, we just need to compute
$\mu_{CZ}(\gamma_3^N)$ and $\mu_{CZ}(\gamma_4^N)$ up to arbitrarily large $N\in\N$ and for arbitrarily small $0<\varepsilon_1,\varepsilon_2<<1$.

For $\gamma_3$, take a vector $\eta_3=(a_1, a_2,1)\in\Z^3$ such that $\{\nu_1 = (v_1,1), \nu_3 = (v_3,1), \eta_3\}$ is a $\ZZ$-basis. This condition is written as 
\[
1=\det(\nu_1, \nu_3, \eta_3)=(1-a_1) p+a_2 q.
\]
Solving the corresponding system, we find out that
\[
R_\varepsilon=b_1\nu_1+b_2\nu_3+b\eta_3
\]
where
\[
b_1=(1-p+a_2)\frac{b}{p} +\varepsilon_1-\frac{(q+1)\varepsilon_2}{p}\,,\ b_2=-a_2 \frac{b}{p} +\frac{\varepsilon_2}{p}
\quad\text{and}\quad b = p(1-\varepsilon_1) + q \varepsilon_2\,.
\]
Therefore, considering also $0<\varepsilon_2 <<\varepsilon_1$ so that $\varepsilon_2 / \varepsilon_1$ is arbitrarily small,
we can apply~(\ref{eq:CZindex}) to get that
\begin{align*}
\mu_{CZ}\left(\gamma_3^N\right) & =2+2\left(\left\lfloor \frac{N b_1}{b}\right\rfloor
+\left\lfloor \frac{Nb_2}{b}\right\rfloor+N \right)\\
&=2+2\left(\left\lfloor -N+\frac{N(1+a_2)}{p}+\varepsilon'\right\rfloor+\left\lfloor \frac{-Na_2}{p}+\varepsilon''\right\rfloor+N\right)\\
&=2+2\left(\left\lfloor \frac{N(1+a_2)}{p}\right\rfloor+\left\lfloor \frac{-Na_2}{p}\right\rfloor\right)
\end{align*}
up to an arbitrarily large $N$.

For $\gamma_4$, one can check easily that the vector $\eta_4=(q, p-1, 1)$ is such that $\{\nu_2 = (v_2,1), \nu_3 = (v_3,1), \eta_4\}$
is a $\ZZ$-basis. Solving the system we get
\[
R_\varepsilon=(-q+\varepsilon_1)\nu_2+\left(-(p-q-1+\varepsilon_1-\varepsilon_2)\right)\nu_3+(p-\varepsilon_2)\eta_4.
\]
Therefore, considering again $0<\varepsilon_2 <<\varepsilon_1$, we can apply~(\ref{eq:CZindex}) to get that
\[
\mu_{CZ}(\gamma_4^N)=2+2\left(\left\lfloor \frac{-Nq}{p}\right\rfloor+
\left\lfloor \frac{N(q+1)}{p}-\varepsilon\right\rfloor\right)
\]
up to an arbitrarily large $N$.

The fundamental group $\pi_1(S^\ast (L^3_p (q)))$ is isomorphic to $\ZZ^3/\mathcal N\cong \ZZ_p$ where 
\[
\mathcal N=\textup{span}_\ZZ\{\nu_0 = (v_0,1), \nu_1=(v_1,1), \nu_2 = (v_2,1), \nu_3 = (v_3,1)\}
\] 
and the isomorphism $\ZZ^3/\mathcal N\cong \ZZ_p$ 
is induced by projection on the second coordinate to $\ZZ$ composed with the projection $\ZZ\to \ZZ_p$, 
$m\mapsto \underline{m}$. 
Via this isomorphism the classes $[\gamma_3], [\gamma_4]$ in the fundamental group correspond respectively to 
$\eta_3+\mathcal N$, $\eta_4+\mathcal N$, which correspond to $\underline a_2\in \ZZ_p$ and $\underline{p-1}\in \ZZ_p$, 
respectively. Let $k\in \{1, \ldots, p-1\}$ and take $r\in \{1, \ldots, p-1\}$ such that $r\equiv kq \mod p$. Since 
$(1-a_1)p+a_2q=1$, 
$q\equiv a_2^{-1}\mod p$, hence $ra_2\equiv k\mod p$. Thus $[\gamma_3^{mp+r}]=[\gamma_4^{(m+1)p-k}]$, for 
$m\in \ZZ_{\geq 0}$, are both the class corresponding to $\underline k\in \ZZ_p\cong\pi_1(M)$. We can compute
\[
\mu_{CZ}\left(\gamma_3^{mp+r}\right)=2+2m+ 2\left( \left\lfloor \frac{k+r}{p} \right\rfloor+\left\lfloor -\frac{k}{p} 
\right\rfloor \right)=\begin{cases}2m &\textup{if }k+r<p\\
2m+2&\textup{if }k+r\geq p\end{cases}
\]
and similarly
\[
\mu_{CZ}\left(\gamma_4^{(m+1)p-k}\right)=4+2m+ 2\left( \left\lfloor \frac{r}{p} \right\rfloor+\left\lfloor 
-\frac{r+k}{p}-\varepsilon \right\rfloor \right) =\begin{cases}2m+2 &\textup{if }k+r<p\\
2m&\textup{if }k+r\geq p\end{cases}
\]
up to an arbitrarily large $m$. Since for contact $5$-manifolds, i.e. $n=2$, the contact homology degree is equal to the
Conley-Zehnder index, for any non-trivial class $\underline k\in (\ZZ_p)\setminus \{0\}\cong \pi_1(S^\ast (L^3_p (q)))\setminus \{0\}$ 
we get that
\[
cb_\ast^{\underline k}(S^\ast (L^3_p (q)))=\begin{cases}1 &\textup{if }\ast=0\\
2 &\textup{if }\ast \geq 2\textup{ is even}\\
0 & \textup{otherwise}\end{cases}
\]
For the trivial class we note that $\deg\left(\gamma_3^{mp}\right)=2+2m$ and $\deg\left(\gamma_4^{mp}\right)=2m$, up 
to an arbitrarily large $m\in \ZZ_{>0}$, and so
\[
cb_\ast^{\underline{0}}(S^\ast (L^3_p (q)))=\begin{cases}1 &\textup{if } \ast=2\\
2 &\textup{if }  \ast> 2\textup{ is even}\\
0 & \textup{otherwise}\end{cases}.
\]
Hence, we get the following contact Betti numbers:
\[
cb_\ast (S^\ast (L^3_p (q)))=\begin{cases}p-1 &\textup{if } \ast=0\\
2p-1 &\textup{if } \ast=2\textup{ is even}\\
2p &\textup{if } \ast\geq 2\textup{ is even}\\
0 & \textup{otherwise}\end{cases}.
\]
Note that these contact Betti numbers do not detect the value of $q$ and do not distinguish any non-trivial class in 
$\pi_1 (S^\ast (L^3_p (q)))$. 

\section{Contact Betti numbers of higher dimensional lens spaces and applications}
\label{s:examples2}

In this section we will use formula~(\ref{eq:CZindex}) for the Conley-Zehnder index to compute the contact Betti numbers 
$cb_\ast$ of Gorenstein contact lens spaces
\[
L^{2n+1}_p (\alpha_0, -\alpha_1, \ldots, -\alpha_{n-1}, 1)
\]
with $n,p\in\N$, $\alpha_1, \ldots, \alpha_{n-1} \in \Z$, $\alpha_0 := \alpha_1 + \cdots + \alpha_{n-1} - 1$ and
$\gcd (\alpha_j , p) = 1$ for $j=0,\ldots, n-1$. 
As illustrative contact topology applications, we will then prove Propositions~\ref{prop:ineq} and~\ref{prop:auto}.
To simplify notation, we will denote these lens spaces by $L^{2n+1}_p (\balpha)$, with 
$\balpha = (-\alpha_1, \ldots, -\alpha_{n-1})\in\Z^{n-1}$, and their contact Betti numbers by $cb_\ast (n,p,\balpha)$ and
$cb_\ast^{\underline{k}} (n,p,\balpha)$, where $\underline{k}\in\Z_p \cong \pi_1 (L^{2n+1}_p (\balpha))$.

\subsection{Contact Betti numbers}

As we saw in \S~\ref{ss:higherlens}, the moment cone $C\subset\R^{n+1}$  of $L^{2n+1}_p (\balpha)$ has defining normals
\[
\nu_0=\begin{bmatrix}
0\\ 
0\\
\vdots\\ 
0\\ 
0\\ 
1
\end{bmatrix},\,\,
 \nu_j=\begin{bmatrix}
0\\ 
\vdots\\ 
1\\ 
0\\
\vdots\\
1
\end{bmatrix}\textup{ for }j=1, \ldots, n-1,\,\, \nu_n=\begin{bmatrix}
\alpha_1\\ 
\alpha_2\\
\vdots\\ 
\alpha_{n-1}\\ 
p\\ 
1
\end{bmatrix}\,.
\]
Let $E_j$ denote the edge of $C$ given by the intersection of all its facets except the one with normal $\nu_j$.
Consider a normalized toric Reeb vector $R_\varepsilon=(\varepsilon_1, \ldots, \varepsilon_{n-1}, \varepsilon_n, 1)$, 
with $0<\varepsilon_1, \ldots, \varepsilon_n<<1$, and denote by $\gamma_j$ its simple 
closed Reeb orbit corresponding to $E_j$, $j=0,\ldots,n$. By taking $\varepsilon_1,\ldots,\varepsilon_n$ such that
 $1,\varepsilon_1,\ldots,\varepsilon_n$ are $\Q$-independent, 
we know that the Reeb flow of $R_\varepsilon$ has exactly these $n+1$ simple closed Reeb orbits and they are all non-degenerate.
Again, either by direct computation using~(\ref{eq:CZindex}) or by using~\cite[Theorem 1.6]{AM2}, we also know that by considering
$\varepsilon_1,\ldots\varepsilon_n$ arbitrarily small we will have that $\mu_{CZ}(\gamma_j^N)$, $j=1,\ldots,n$, are arbitrarily large for 
any $N\in\N$. Hence, we just need to compute $\mu_{CZ}(\gamma_0^N)$ up to arbitrarily large $N\in\N$ and for arbitrarily small $0<\varepsilon_1,\ldots, \varepsilon_n<<1$.

Take a vector $\eta_0=(a_1, \ldots, a_{n-1}, a_n, 1)\in\Z^{n+1}$ that completes $\nu_1, \ldots, \nu_{n}$ to a $\ZZ^{n+1}$-basis, 
that is, such that the following equation holds:
\begin{align*}
1&=\begin{vmatrix}
1 & 0 &    & 0 & \alpha_1 & a_1\\ 
0 & 1 & \ldots  & 0 & \alpha_2 &a_2 \\ 
\vdots & \vdots &    & \vdots & \vdots &\vdots \\  
0 & 0 &  &   1 & \alpha_{n-1} & a_{n-1}\\ 
0 & 0 &  \ldots  & 0 & p & a_n\\ 
1 & 1 &    & 1 & 1 &1 
\end{vmatrix} \\
& =
p \begin{vmatrix}
1 & 0 &    & 0 & a_1 \\ 
0 & 1 & \ldots  & 0 & a_2 \\ 
\vdots & \vdots &    & \vdots & \vdots  \\  
0 & 0 & \ldots &   1 & a_{n-1} \\ 
1 & 1 &    & 1 & 1 
\end{vmatrix}
-
a_n\begin{vmatrix}
1 & 0 &    & 0 & \alpha_1 \\ 
0 & 1 & \ldots  & 0 & \alpha_2 \\ 
\vdots & \vdots &    & \vdots & \vdots  \\  
0 & 0 & \ldots &   1 & \alpha_{n-1} \\ 
1 & 1 &    & 1 & 1
\end{vmatrix}\\
&=-p ( 1 - a_1-\ldots-a_{n-1})-a_n(1-\alpha_1-\ldots-\alpha_{n-1})\\
&=\left(1-\sum_{j=1}^{n-1}a_j\right)p+\alpha_0 a_n\,.
\end{align*}
Thus we can take $a_n\equiv \alpha_0^{-1} \mod p$ and $a_j$, $j=1,\ldots, n-1$, in a way that the above equality holds.

We now want to write $R_\varepsilon$ in the basis $\{ \nu_1, \ldots, \nu_{n}, \eta_0\}$, i.e. we want to find $b_j$,
$j=1, \ldots, n$, and $b$ such that 
\[
R_\varepsilon = b_1\nu_1+\ldots+b_n\nu_n + b\eta_0\,.
\]
We get the following system
\[
\begin{cases}
\varepsilon_j=b_j+\alpha_j b_n+a_j b & j=1, \ldots, n-1\\ 
\varepsilon_n=p b_n+a_n b\\ 
\sum_{j=1}^{n} b_j + b = 1\,.
\end{cases}
\]
The solution for this system is given by
\[
b_n =-\frac{b a_n}{p}+\frac{\varepsilon_n}{p}\,,\ \ 
b_j =\frac{ \alpha_j a_n b}{p}-a_j b +\varepsilon_j-\frac{\alpha_j\varepsilon_n}{p}\,, 1\leq j \leq n-1\,, 
\]
\[
\textup{ and }\  \ b = p \left(1-\sum_{j=1}^{n-1} \varepsilon_j \right) + \alpha_0 \varepsilon_n.
\]
Therefore, considering also $0<\varepsilon_n <<\varepsilon_j$ so that $\varepsilon_n / \varepsilon_j$ is arbitrarily small,
for all $j=1,\ldots, n-1$, we can apply~(\ref{eq:CZindex}) to get that
\begin{align*}
\mu_{CZ}\left(\gamma_0^N\right) 
& = n + 2\left( \sum_{j=1}^{n-1} \left\lfloor \frac{N b_j}{b}\right\rfloor +\left\lfloor \frac{Nb_n}{b}\right\rfloor+N \right)\\
& = n + 2\left( \sum_{j=1}^{n-1} \left\lfloor{\frac{N\alpha_ja_n}{p}-Na_j} + \varepsilon'_j \right \rfloor
+ \left \lfloor{-\frac{Na_n}{p}}  + \varepsilon'_n \right \rfloor+N \right)\\
& = n + 2\left(\sum_{j=1}^{n-1}\left \lfloor{\frac{N\alpha_ja_n}{p}}\right \rfloor+\left \lfloor{-\frac{Na_n}{p}}\right \rfloor
+N\left(1-\sum_{j=1}^{n-1}a_j\right)\right)\\
& = n + 2\left(\sum_{j=1}^{n-1}\left \lfloor{\frac{N\alpha_ja_n}{p}}\right \rfloor+\left \lfloor{-\frac{Na_n}{p}}\right \rfloor
+N\left(\frac{1-\alpha_0a_n}{p}\right) \right)\\
&= n + 2\left( \frac{N}{p}-\sum_{j=1}^{n-1}\left\{ \frac{N\alpha_ja_n}{p}\right \}-\left \{-\frac{Na_n}{p}\right \}\right)
\end{align*}
up to an arbitrarily large $N$. Recall that $\{x\} = x - \lfloor x \rfloor$ for any $x\in\R$.

Let us define the (SFT) degree function $g: \N \to \Z$ by
\begin{equation} \label{eq:deg}
g(N)\equiv \mu_{CZ}(\gamma_0^N)+n-2 =
2\left( \frac{N}{p} + (n-1) -\sum_{j=1}^{n-1}\left\{ \frac{N\alpha_ja_n}{p}\right \}-\left \{-\frac{Na_n}{p}\right \}\right)\,.
\end{equation}
We can observe the following:
\begin{itemize}
\item $g(N)\geq 0$ for every $N>0$;
\item $g(p)=2n$;
\item $g(N)\leq 2(n-1)$ for $N<p$;
\item $g(N+p)=g(N)+2$ for every $N>0$.
\end{itemize}
By the last property, it is clear that
$$cb_{2j} (n,p,\balpha) =\#\left\{N\in \{1, \ldots, p\}: g(N)\leq 2j\right\}$$
since the congruence class $\{kp+N: k\in \ZZ^{+}_0\}$ contributes to the homology of degree $2j$ with $1$ if $g(N)\leq 2j$ 
and with $0$ otherwise. In particular
$$cb_{2(n-1)}  (n,p,\balpha) =p-1 \quad\textup{ and }\quad cb_{2j} (n,p,\balpha) =p \textup{ if } j \geq n\,.$$
Note also that, since $\left\{x\right\}+\left\{-x\right\}=1$ if $x\notin \ZZ$, we have $g(N)+g(p-N)=2(n-1)$ for $N=1, \ldots, p-1$. 
Thus 
$$g(N)\leq 2j \quad\textup{ if and only if} \quad g(p-N)>2(n-2-j)$$
and the following symmetry property follows:
$$cb_{2j}  (n,p,\balpha) +cb_{2(n-2-j)}  (n,p,\balpha) =p-1 \quad\textup{ for }j=0,1, \ldots, n-2.$$
For information on the decomposition induced by homotopy classes in the cyclic fundamental group of these Gorenstein
contact lens spaces, note that $\gamma_0$ represents a generator and we get that
\[
cb_{2j}^{[\gamma_0^N]} (n,p,\balpha) =\begin{cases}1 & \textup{if 2}j\geq \widetilde g(N) \\
0 &\textup{otherwise,}\end{cases}
\]
where $\widetilde{g}: \N \to \Z$ is defined by
\begin{equation} \label{eq:degp}
\widetilde{g} (N)=g(N) \quad\textup{if } N=1,\ldots,p\,, \quad\textup{and}\quad \widetilde{g} (N+p) = \widetilde{g} (N) 
\quad\textup{for all } N\in\N\,.
\end{equation}
We summarize these facts in the following proposition.
\begin{prop} \label{prop:lensBetti}
We have that $cb_j (n,p,\balpha) = 0$ whenever $j$ is odd or $j<0$ and
\[
cb_{2j} (n,p,\balpha) = \#\left\{N\in \{1, \ldots, p\}: g(N)\leq 2j\right\} \ \text{for all $j\geq 0$,}
\]
where $g$ is the degree function defined by~(\ref{eq:deg}).
In particular:
\begin{itemize}
\item[(i)] $cb_{2j} (n,p,\balpha) = p$ for all $j\geq n$;
\item[(ii)] $cb_{2(n-1)} (n,p,\balpha) = p-1$;
\item[(iii)] $cb_{2j} (n,p,\balpha) + cb_{2(n-2-j)} (n,p,\balpha) = p-1$ for $j=0,1,\ldots, n-2$.
\end{itemize}
Moreover,
\[
cb_{2j}^{[\gamma_0^N]} (n,p,\balpha) =\begin{cases}1 & \textup{if 2}j\geq \widetilde g(N) \\
0 &\textup{otherwise,}\end{cases}
\]
where $\widetilde{g}$ is defined by~(\ref{eq:degp}). 
\end{prop}
When the decomposition induced by homotopy classes in the cyclic fundamental group is not relevant,
the following definition provides a useful compact way of presenting the relevant information provided by the 
contact Betti numbers of a Gorenstein contact lens space.
\begin{defn} \label{def:cbs}
The \emph{contact Betti numbers sequence} of a Gorenstein contact lens space $L^{2n+1}_p (\balpha)$ is defined as
\[
cbs (n,p,\balpha) := \left( cb_0 (n,p,\balpha), cb_2 (n,p,\balpha), \ldots, 
cb_{2(n-1)} (n,p,\balpha) \right) \in \N_0^{n}\,.
\]
\end{defn}

\begin{example}
When $n=2$ we have that $p$ must be odd and
\[
cbs (2,p,\balpha)  = \left(\frac{p-1}{2}, p-1\right)\,,
\]
i.e. for $5$-dimensional Gorenstein contact lens spaces $L^{5}_p (\balpha)$ the contact Betti numbers sequence is independent 
of the weight $\balpha = (\alpha_1)$.
\end{example}

\begin{example}
To see that when $n\geq 3$ the contact Betti numbers sequence does depend on $\balpha$, let us compute it for
$L^7_5 (-3,1,1,1)$ and $L^7_5 (1,-1,-1,1)$.

For $L^7_5 (-3,1,1,1)$ we have that $n=3$, $p=5$, $\alpha_1=\alpha_2=-1$. Then $\alpha_0=-3$ and we can take 
$a_n=3\equiv_5 (-3)^{-1}$ in~(\ref{eq:deg}) to get that the degree function is given by 
\[
g(N)=2\left(\frac{N}{5} + 2  - 3\left\{ \frac{-3N}{5}\right\} \right)\,.
\]
This implies that $g(1)=g(4)=2$, $g(2)=0$, $g(3)=4$ and $g(5)=6$. It follows that the contact Betti numbers sequence is
\[
cbs (3,5,(-1,-1))  = \left(1,3,4\right)\,.
\]

For $L^7_5 (1,-1,-1,1)$ we have that $n=3$, $p=5$, $\alpha_1=\alpha_2=1$. Then $\alpha_0=1$ and we can take 
$a_n=1\equiv_5 1^{-1}$ in~(\ref{eq:deg}) to get that the degree function is given by 
\[
g(N)=2\left(\frac{N}{5} + 2  - 2\left\{ \frac{N}{5} \right\} - \left\{ -\frac{N}{5} \right\} \right)\,.
\]
This implies that $g(1)=g(2)=g(3)=g(4)=2$ and $g(5)=6$. It follows that the contact Betti numbers sequence is
\[
cbs (3,5,(1,1))  = \left(0,4,4\right)\,.
\]
\end{example}

\begin{example} \label{ex:preq}
When $n+1 = k p$ for some $k\in\N$, the contact lens space 
$$L_p^{2n+1}(1,\ldots, 1)$$ 
is Gorenstein. Indeed, we can take $\alpha_j=-1$, 
$j=1, \ldots, n-1$, $\alpha_0 = -n =1-(n+1)\equiv_{p} 1$ and $a_n=1$. The degree function~(\ref{eq:deg}) is given by
\[
g(N)=2\left(\frac{N}{p}+(n-1)-n\left\{-\frac{N}{p}\right\}\right)\,.
\]
When $N\in \{1,2,\ldots, p-1\}$ we get 
\[
g(N)=2\left(\frac{N}{p}+(n-1)-n\left(1-\frac{N}{p}\right)\right)=2\left(\frac{N(n+1)}{p}-1\right)=2(kN-1)\,.
\]
It follows that 
\[
cb_{2\ast} (kp-1, p, (-1,\ldots,-1)) = j 
\]
whenever $j=1,\ldots,p-1$ and $kj-1\leq \ast < k (j+1) -1$.
For example, this implies that the contact Betti numbers sequence of Gorenstein real projective spaces, i.e. 
$\rp^{2n+1}$  with $n$ odd, is
\[
cbs (n, 2, (-1,\ldots,-1)) = (0,\ldots,0,1, \ldots, 1)\ \text{with $(n-1)/2$ zeros and $(n+1)/2$ ones.} 
\]
It also implies that the contact Betti numbers sequence of the prequantization of $(\cp^n, [\omega] = 2\pi c_1 (\cp^n))$, i.e.
$L_p^{2n+1}(1,\ldots, 1)$ with $p=n+1$ (hence $k=1$), is
\[
cbs (n, n+1, (-1,\ldots,-1)) = (1,2, \ldots, n)\,.
\]
\end{example}

\subsection{Applications}

We will now prove Propositions~\ref{prop:ineq} and~\ref{prop:auto}.

To prove Proposition~\ref{prop:ineq}, we will show that the Gorenstein contact lens spaces 
\[
L^{13}_5 (1,1,1,1,2,-2,1)\quad\text{and}\quad L^{13}_5 (1,-1,-1,-1,-2,-2,1)\,,
\]
which are diffeomorphic as naturally oriented manifolds (they have the same weights up to an even number of sign changes), 
have different contact Betti numbers sequences. 

For $L^{13}_5 (1,1,1,1,2,-2,1)$ 
we have that $n=6$, $p=5$, $\balpha=(-1,-1,-1,-2,2)$ and  $\alpha_0=-4 \equiv_5 1$. We 
can take $a_6=1\equiv_5 1^{-1}$ in~(\ref{eq:deg}) to get that the degree function is given by 
\[
g(N)=2\left(\frac{N}{5} + 5  - 4\left\{ \frac{-N}{5}\right\}  -  \left\{ \frac{-2N}{5}\right\} -  \left\{ \frac{2N}{5}\right\}  \right)\,.
\]
This implies that $g(1)=2$, $g(2)=4$, $g(3)=6$,  $g(4)=8$ and $g(5)=12$. It follows that the contact Betti numbers sequence is
\[
cbs (6,5,(-1,-1,-1,-2,2))  = \left(0,1,2,3,4,4\right)\,.
\]
For $L^{13}_5 (1,-1,-1,-1,-2,-2,1)$ we have that $n=6$, $p=5$, $\balpha=(1,1,1,2,2)$ and  $\alpha_0=6 \equiv_5 1$. We 
can take $a_6=1\equiv_5 1^{-1}$ in~(\ref{eq:deg}) to get that the degree function is given by 
\[
g(N)=2\left(\frac{N}{5} + 5  - 3\left\{ \frac{N}{5}\right\}  - 2 \left\{ \frac{2N}{5}\right\} -  \left\{ \frac{-N}{5}\right\}  \right)\,.
\]
This implies that $g(1)=g(3) = 6$, $g(2)=g(4)=4$ and $g(5)=12$. It follows that the contact Betti numbers sequence is
\[
cbs (6,5,(1,1,1,2,2))  = \left(0,0,2,4,4,4\right)\,.
\]

To prove Proposition~\ref{prop:auto}, we will use the decomposition of the contact Betti numbers sequence by homotopy 
classes in the fundamental group to show that 
$$L^{15}_5 (1,1,1,2,-2,-2,-2,1)$$ 
is an example of a Gorenstein contact lens space 
with an homotopy class of orientation preserving diffeomorphisms that does not contain any contactomorphism. Before
analysing this particular example in detail, let us make some more general considerations.

Let $L_p$ denote a lens space with $\pi_1(L_p)\cong \ZZ_p$. Any diffeomorphism $f: L_p\to L_p$ induces an automorphism 
$f_\ast:\ZZ_p\to\ZZ_p$. Such an automorphism must be multiplication by $k$ for some $k\in (\ZZ_p)^\ast$. 
It follows from the last part of Proposition~\ref{prop:lensBetti} that if such an automorphism is induced by a co-orientation 
preserving contactomorphism of $L_p$ then $\widetilde g(kN)=\widetilde g(N)$ for every $N\in\N$, where $\widetilde{g}$ is 
defined by~(\ref{eq:degp}). This is a restrictive condition. As the examples above and in the previous subsection illustrate 
well, it often implies that $k=1$. However, it is also somewhat restrictive to find lens spaces $L_p$ with non-trivial 
automorphisms of $\ZZ_p$ induced by orientation preserving diffeomorphisms (see~\cite{M} and~\cite{HJ}). In fact, for $p$ 
odd, the automorphisms induced by orientation preserving homotopy equivalences correspond to multiplication by $k\in (\ZZ_p)^\ast$ 
such that $k^{n+1} \equiv 1 \mod p$, where $\dim L_p = 2n+1$. To characterize the homotopy equivalences that can be 
represented by diffeomorphisms, let $\ell_0, \ell _1, \ldots, \ell_n \in\Z$ be the weights of $L_p$ and consider
 $\Delta(L_p)\in \QQ[\ZZ_p]$ defined by 
\[
\Delta(L_p) (t) =\prod_{j=0}^n \left(t^{r_j}-1\right)\in \QQ[\ZZ_p]\,,
\]
where $t$ is a (multiplicative) generator of $\ZZ_p$ and $r_j \ell_j \equiv 1 \mod p$.
Any automorphism of $\ZZ_p$ extends to an automorphism of $\QQ[\ZZ_p]$ and the class $[f]$ of homotopy equivalences can 
be represented by an orientation preserving diffeomorphism if and only if $f_\ast \Delta(L_p)= t^u \Delta(L_p)$ for some 
$u\in \ZZ$ (see~\cite[Proposition 3.3]{HJ}).

Consider now the Gorenstein contact lens space $L^{15}_5 (1,1,1,2,-2,-2,-2,1)$. We have that $n=7$, $p=5$, 
$\balpha=(-1,-1,-2,2,2,2)$ and  $\alpha_0=1$. We can take $a_7=1\equiv_5 1^{-1}$ in~(\ref{eq:deg}) to get that the degree 
function is given by 
\[
g(N)=2\left(\frac{N}{5} + 6  - 3\left\{ \frac{2N}{5}\right\}  - 3 \left\{ \frac{-N}{5}\right\} -  \left\{ \frac{-2N}{5}\right\}  \right)\,.
\]
This implies that $g(1)=g(2) = 4$, $g(3)=g(4)=8$ and $g(5)=14$, which means that there is no $1 \ne k\in (\ZZ_5)^\ast$ such
that $\widetilde g(kN)=\widetilde g(N)$ for every $N\in\N$. Hence, any co-orientation preserving contactomorphism of
$L^{15}_5 (1,1,1,,2,-2,-2,-2,1)$ acts trivially on its fundamental group. Since the contactomorphism induced by complex
conjugation on all coordinates reverses the co-orientation and acts by $-1$ on the fundamental group, we conclude that
any contactomorphism acts by multiplication by $\pm 1 \in (\ZZ_5)^\ast$. However, this lens space admits orientation 
preserving diffeomorphisms 
that act by multiplication by $2\in (\ZZ_5)^\ast$ on its fundamental group. In fact note that $2^8 = 256 \equiv_5 1$, while
\[
\Delta (t) = (t-1)^4 (t^3 -1) (t^{-3} -1)^3 = \frac{(-1)^3}{t^4} (t-1)^4 (t^3 -1)^4
\]
and
\[
\Delta (t^2) = (t^2-1)^4 (t -1) (t^{-1} -1)^3 = (-1)^7 (t-1)^4 (t^3 -1)^4 = t^4 \Delta (t) \,.
\]
An explicit example of such a diffeomorphism is the one induced by the following linear automorphism of $\C^8$:
\[
(z_0, z_1, z_2, z_3, z_4, z_5, z_6, z_7) \mapsto (\bar z_3, z_4, z_5, z_7, \bar z_0, \bar z_1, \bar z_2, z_6)\,.
\]

\appendix

\section{Total Chern class of lens spaces} 
\label{appendix}

The following proposition is elementary and certainly well known, although we could not find an explicit
reference to it in the literature.
\begin{proposition} \label{prop:totalChern}
Under natural isomorphisms $H^{2j} (L^{2n+1}_p (\ell_0, \ell _1, \ldots, \ell_n); \Z) \cong \Z_p$, for $j=1,\ldots,n$, 
the Chern classes of the canonical contact structure $\xi$ on 
$$L^{2n+1}_p (\ell_0, \ell _1, \ldots, \ell_n)$$ 
are given by
\[
c_j = \sigma_j \!\!\!\!  \mod p\,, j=1,\ldots n\,,
\]
where $\sigma_j =$ $j$-th elementary symmetric polynomial of 
$\ell_0, \ell _1, \ldots, \ell_n$.
\end{proposition}
\begin{proof}
This is a particular case of Proposition 2.16 in~\cite{AM1} and its proof, together with the fact that
the $j$-th Chern class of a sum of line bundles is the $j$-th elementary symmetric polynomial of
their first Chern classes. The main points are the following:
\begin{itemize}
\item[(i)] The quotient map from $S^{2n+1}$ to $L_p^{2n+1} (\ell_0, \ell_1, \ldots, \ell_n)$ is a principal $\Z_p$-bundle 
and its classifying map $f: L_p^{2n+1} (\ell_0, \ell_1, \ldots, \ell_n) \to B \Z_p$ induces an isomorphism 
$f^\ast : H^2 (B \Z_p; \Z) \to H^2 (L_p^{2n+1} (\ell_0, \ell_1, \ldots, \ell_n); \Z)$.
\item[(ii)] The $\Z_p$-action on $\C^{n+1}$ gives rise to an associated vector bundle $S^{2n+1} \times_{\Z_p} \C^{n+1}$ 
over $L_p^{2n+1} (\ell_0, \ell_1, \ldots, \ell_n)$, and
\[
c_1 (\xi) = c_1 (S^{2n+1} \times_{\Z_p} \C^{n+1}) = f^\ast c_1 (E\Z_p \times_{\Z_p}  \C^{n+1}).
\]
\item[(iii)] $H^2 (B \Z_p; \Z) \cong \Z_p \cong$ character group of $\Z_p$ and $c_1 (E\Z_p \times_{\Z_p}  \C^{n+1})$ 
is the sum of the characters that determine the $\Z_p$-action on $\C^{n+1}$.
\end{itemize}
\end{proof}

Consider the polynomial $c(\xi) \in \Z_p [x]$ given by
\[
c(\xi) = \prod_{i=0}^n \left( 1 + \ell_i x \right) = 1 + \prod_{i=1}^{n+1} \sigma_i x^i \in \Z_p [x]\,.
\]
The total Chern class of $\xi$ is given by $[c(\xi)] \in \Z_p [x]/x^{n+1}$ and Proposition~\ref{prop:totalChern} says that $\xi$ 
has total Chern class equal to zero if and only if we have that
\[
\prod_{i=0}^n \left(1 + \ell_i x \right)  = 1 + \sigma_{n+1} x^{n+1} = 1 + \left(\prod_{i=0}^n \ell_i\right) x^{n+1} \in \Z_p [x]\,.
\]
This equation is quite restrictive but it does have some solutions that provide a few interesting examples of lens spaces
with canonical contact structure with zero total Chern class. For lens spaces of the form
\[
L_p^{2n+1} (1, 1, \ldots, 1)\,,\ \text{i.e.} \ \ell_0 = \ell_1 = \cdots = \ell_n = 1\,,
\]
we have that 
\[
[c(\xi)] = 0 \Leftrightarrow (1+x)^{n+1} = 1 + x^{n+1} \in \Z_p [x] \Leftrightarrow n+1 = p^k \ \text{for $p$ prime and some $k\in\N$.} 
\]
Hence, any lens space of the form $L_p^{2p^k -1} (1, 1, \ldots, 1)$, with $p$ prime, has zero total Chern class. Another set 
of examples arises from the fact that, for prime $p$, we have
\[
1 - x^{p-1} = \prod_{i=0}^{p-2} (1 + (i+1) x) \in \Z_p [x]\,.
\]
Hence, any lens space of the form $L_p^{2p-3} (1, 2, \ldots, p-1)$, with $p$ prime, also has zero total Chern class. 

Despite these and some other few examples of lens spaces with canonical contact structure with zero total Chern class, 
they do not provide relevant examples for the type of results considered in Propositions~\ref{prop:ineq} and~\ref{prop:auto}.
In fact, at least for prime $p$, there are no examples of lens spaces where the  type of results considered 
in Propositions~\ref{prop:ineq} and~\ref{prop:auto} do not follow from total Chern class considerations, the reason being
that, as we will now see, the canonical contact structure of a lens space $L^{2n+1}_p (\ell_0, \ell _1, \ldots, \ell_n)$  is 
completely determined by the diffeomorphism type of the lens space and the total Chern class of the contact structure.

\begin{proposition}
Let $p$ be a prime and $f$ be an orientation preserving diffeomorphism of $L_p=L_p^{2n+1}(\ell_0, \ldots, \ell_n)$. 
Suppose that $f_\ast: \pi_1(L_p)\to \pi_1(L_p)$ is multiplication by $k\in \ZZ_p^\ast$.
If $f$ preserves the total Chern class of the canonical contact structure of $L_p$, then there is a permutation $\tau$ of 
the set $\{0, 1, \ldots, n\}$ such that
\[\ell_{\tau(i)}\equiv k \ell_i \mod p.\]
 In particular,  $\tilde{g}$ defined by (\ref{eq:degp})  satisfies $\tilde g(kN)=\tilde g(N)$ for every $N$, so the contact Betti
 numbers and their decomposition by homotopy classes of the fundamental group do not impose any restriction for $f$ to 
 be homotopic to a contactomorphism of the canonical contact structure of $L_p$.
\end{proposition}
\begin{proof}
The condition that $f$ preserves the total Chern class is equivalent to $k^i \sigma_i \equiv \sigma_i \mod p$ for $i=1, \ldots, n$. 
Moreover, since $f$ is an orientation preserving diffeomorphism we also have $k^{n+1}\equiv 1 \mod p$. Hence the congruence 
above 
also holds for $i=n+1$. Therefore 
\[
\prod_{i=0}^n \left( 1 + \ell_i x \right) = 1 + \prod_{i=1}^{n+1} \sigma_i x^i  = 1 + \prod_{i=1}^{n+1} k^i \sigma_i x^i =
\prod_{i=0}^n \left( 1 + k \ell_i x \right) \,.
\]
The existence of the permutation $\tau$ follows from $\ZZ_p$ being a field.\qedhere
\end{proof}

\begin{proposition}
Let $p$ be a prime. Let $f$ be an orientation preserving diffeomorphism
\[f\colon L_p=L_p^{2n+1}(\ell_0, \ldots, \ell_n)\to L_p^{2n+1}(\ell_0', \ldots, \ell_n')=L_p'.\]
If $f$ preserves the total Chern class of the corresponding canonical contact structures $\xi$ and $\xi'$, 
i.e. $f^\ast [c(\xi')] = [c(\xi)]  \in \Z_p [x]/x^{n+1}$, then there is $k\in \ZZ_p^\ast$ and a permutation $\tau$ 
of the set $\{0, 1, \ldots, n\}$ such that
\[\ell_{\tau(i)}'\equiv k \ell_i \mod p.\]
\end{proposition}
\begin{proof}
Exactly the same as the proof of the previous proposition once we fix isomorphisms $\pi_1(L_p)\cong \ZZ_p \cong \pi_1(L_p').$
\end{proof}

\begin{cor}
Let $p$ be a prime. Let $f$ be an orientation preserving diffeomorphism
\[f\colon L_p=L_p^{2n+1}(\ell_0, \ldots, \ell_n)\to L_p^{2n+1}(\ell_0', \ldots, \ell_n')=L_p'.\]
Then $f$ is homotopic to a contactomorphism of the corresponding canonical contact structures $\xi$ and $\xi'$
if and only if $f^\ast [c(\xi')] = [c(\xi)]  \in \Z_p [x]/x^{n+1}$.
\end{cor}
\begin{proof}
This follows from the previous proposition together with the fact that diffeomorphisms of lens spaces are homotopic if and only if 
they induce the same map on the fundamental group (because inducing the same map on $\pi_1$ implies they induce the same 
map on cohomology due to its ring structure; see also Proposition 3.2 in~\cite{HJ}). 
\end{proof}

\section*{Acknowledgements} 
We are grateful to Renato Vianna for asking the first named author for a toric description of unit cosphere 
bundles of $3$-dimensional lens spaces. The answer is given in Theorem~\ref{thm:lensDiag} and provides
one of the two sets of examples discussed in detail in this paper. We are also grateful to one referee
for making us look at the total Chern class of the canonical contact structures on lens spaces and to Gustavo
Granja for helping us with that.

\end{document}